\documentclass[10pt,leqno]{amsart}
\usepackage{amsmath,amsthm}
\usepackage{amssymb}
\usepackage[all]{xypic}
\usepackage{amsbsy}
\usepackage[colorlinks=true, linkcolor=black, citecolor=blue, urlcolor=blue]{hyperref}
\usepackage{fixltx2e,graphicx}
\usepackage{float}

\theoremstyle{plain}
\newtheorem{thm}[equation]{Theorem}
\newtheorem{prop}[equation]{Proposition}
\newtheorem{lem}[equation]{Lemma}
\newtheorem{cor}[equation]{Corollary}

\theoremstyle{definition}
\newtheorem{defn}[equation]{Definition}
\newtheorem{rem}[equation]{Remark}

\def\defeq{\overset{\mathrm{def}}=}

\def\goth{\mathfrak}

\def\Ga{\Gamma}

\def\cF{\mathcal{F}}

\def\FF{\mathbb{F}}
\def\GG{\mathbb{G}}

\def\SS{\mathbb{S}}

\def\Z{\mathbb{Z}}
\def\ZZ{{{\Z}}}

\def\Aut{\mathrm{Aut}}
\def\End{{{\mathrm{End}}}}
\def\Ext{\mathrm{Ext}}
\def\Gal{\mathrm{Gal}}
\def\Hom{\mathrm{Hom}}

\def\Tor{\mathrm{Tor}}
\def\colim{\mathrm{colim}}

\def\holim{\mathrm{holim}}
\def\lim{\mathrm{lim}}

\def\longr{{{\longrightarrow\ }}}

\def\map{{{\mathrm{map}}}}
\def\det{{{\mathrm{det}}}}

\numberwithin{equation}{section}
\def\ZZ{{{\mathbb{Z}}}}
\def\FF{{{\mathbb{F}}}}
\def\mm{{{\mathfrak{m}}}}

\def\Hom{{{\mathrm{Hom}}}}
\def\Ext{{{\mathrm{Ext}}}}

\def\WW{{{\mathbb{W}}}}
\def\KK{{\mathbb{K}}}
\def\SS{{{\mathbb{S}}}}

\def\kappabar{\overline{\kappa}} 
\usepackage[pdftex,dvipsnames]{xcolor}
\usepackage{xargs} 
\usepackage{todonotes}

\date{}
\begin{document}
\title[Topological resolutions in $K(2)$-local homotopy theory]{Topological resolutions in $K(2)$-local
homotopy theory at the prime $2$}
\address{School of Mathematics, 
Institute for Advanced Study, Princeton, NJ 08540, U.S.A.}
\email{ibobkova@math.ias.edu}

\address{Department of Mathematics, 
Northwestern University, Evanston, IL 60208, U.S.A.}
\email{pgoerss@math.northwestern.edu}  
\begin{abstract}
We provide a topological duality resolution for the spectrum $E_2^{h\SS_2^1}$, which itself
can be used to build the $K(2)$-local sphere. The resolution is built from spectra of 
the form $E_2^{hF}$ where $E_2$ is the Morava spectrum for the formal group
of a supersingular curve at the prime $2$ and $F$ is a finite subgroup of the automorphisms 
of that formal group. The results are in complete analogy with
the resolutions of \cite{GHMR} at the prime $3$, but the methods are of necessity
very different.  As in the prime $3$ case, the main difficulty is in identifying the top
fiber; to do this, we make calculations using
Henn's centralizer resolution.
\end{abstract}

\author{Irina Bobkova}
\author{Paul G. Goerss}
\thanks{The second author was partially supported by the National Science Foundation. A portion of the work
was done at the Hausdorff Institute of Mathematics, during the Trimester Program ``Homotopy Theory, 
Manifolds, and Field Theories.'' }

\maketitle

Chromatic stable homotopy theory uses the algebraic geometry of smooth one-parameter formal groups
to organize calculations and the search for large scale phenomena. In particular, the chromatic  filtration on
the category of the $p$-local finite spectra corresponds to the height filtration for formal groups. The layers
of the chromatic filtration are given by localization with respect to the Morava $K$-theories $K(n)$, with
$n \geq 0.$ Thus, to understand the homotopy type of a finite spectrum $X$ we begin by addressing $L_{K(n)}X$ for 
all prime numbers  $p$ and all $0 \leq n < \infty.$ A useful  and inspirational 
guide to this point of view can be found in the table in section 2 of \cite{HG}. 

If $n=0,$ $K(0)=H\mathbb{Q}$ and $L_0X$ is the rational homotopy type of $X.$
For $n\geq 1,$ the basic computational tool in $K(n)$-local homotopy theory is the $K(n)$-local
Adams-Novikov Spectral Sequence
$$
H^s(\GG_n,(E_n)_tX) \Longrightarrow \pi_{t-s}L_{K(n)}X.
$$
Here $\GG_n$ is the automorphism group of a pair $(\FF_q,\Gamma_n)$ where $\FF_q$ is a finite field
of characteristic $p$ and $\Gamma_n$ is a chosen formal group of height $n$ over $\FF_q$. Then
$E_n$ is the Morava (or Lubin-Tate) $E$-theory defined by $(\FF_q,\Gamma_n)$. We will give
more details and make precise choices in section \ref{ch:Knstuff}. 

First suppose $p$ is large with respect to $n$. (To be precise, we may take $2p-2 > \mathrm{max}\{n^2,2n+2\}$ or
$p > 3 $ if $n$=2.) Then the  Adams-Novikov Spectral Sequence for $X=S^0$
collapses and will have no extensions, so the problem becomes algebraic, although by no means easy. 
See for example, \cite{ShimomuraLarge}, \cite{BehrensRevisited}, or \cite{Lader}, for the case $n=2$ and $p > 3$. 
However, if $p$ is small with respect to $n$, the group $\GG_n$ has finite subgroups of $p$-power order
and the spectral sequence will usually have differentials and extensions, so the problem is no longer 
purely algebraic. At this point, topological resolutions become a useful way to organize the contributions
of the finite subgroups.  The key to unlocking this idea is the Hopkins-Miller theorem, which implies
that $\GG_n$, and hence all of its finite subgroups, act on $E_n$ and that $L_{K(n)}S^0 \simeq E_n^{h\GG_n}$.
See \cite{DH} for this and more. 

The prototypical example is at $n=1$ and $p=2$. Adams and Baird \cite{B}, and Ravenel \cite{R} showed that 
here we have a fiber sequence
$$
L_{K(1)}S^0 \rightarrow KO \xrightarrow{\psi^3-1} KO 
$$
where $KO $ is 2-complete real $K$-theory. For a suitable choice of a height one formal group $\Gamma_1$ we can 
take $E_1 = K$, where $K$ is $2$-complete complex $K$-theory. Then
 $\GG_1 = Aut(\Gamma_1, \FF_2) \cong \ZZ_2^\times$ is the units in the $2$-adic integers, and $C_2 = \{\pm 1\}
\subseteq \ZZ_2^\times$ acts through complex conjugation. We can then rewrite
this fiber sequence as
\[
L_{K(1)}S^0 = E_1^{h\mathbb{G}_1}\longr E_1^{hC_2} \xrightarrow{\psi^3-1} E_1^{hC_2},
\]
and $\psi^3$ is a topological generator of $\ZZ_2^{\times}/\{\pm 1\}\simeq \ZZ_2$. 

For higher heights the topological resolutions will not be simple fiber sequences, but finite towers
of fibrations with the successive fibers built from  $E^{hF_i}$ where $F_i$ runs over various
finite subgroups of $\GG_n$.  

In~\cite{GHMR} the authors generalized the fiber sequence of the $K(1)$-local case to the case $n=2$ and $p=3$. 
One way to say what they proved is the following. Let us write $E=E_2$ to simplify notation.
First we have a split short exact sequence of groups
$$
\xymatrix{
\{1\} \to \GG_2^1 \rto & \GG_2 \rto^-N & \ZZ_3 \to \{1\}
}
$$
where $N$ is obtained from a determinant (See (\ref{norm-defined})) and hence a fiber sequence
\[
L_{K(2)}S^0 \longr E^{h\mathbb{G}_2^1}  \xrightarrow{\psi-1} E^{h\mathbb{G}_2^1},
\]
where $\psi$ is any element of $\mathbb{G}_2$ which maps by $N$ to a topological generator of $\mathbb{Z}_3$.
Then, second, there exists a resolution of $E^{h\GG_2^1}$
$$
E^{h\mathbb{G}_2^1} \to E^{hG_{24}} \to \Sigma^8 E^{hSD_{16}} \to \Sigma^{40} E^{hSD_{16}} \to \Sigma^{48}
E^{hG_{24}}
$$
Here {\it resolution} means each successive composition is null-homotopic and all possible Toda brackets
are zero modulo indeterminacy; thus, the sequence refines to a tower of fibrations
\begin{equation*} \label{eq:3tower}
\xymatrix{
E^{h\mathbb{G}_2^1} \ar[r] & X_2 \ar[r] & X_1 \ar[r] &E^{hG_{24}}\\
\Sigma^{45}E^{hG_{24}} \ar[u] & \Sigma^{38}E^{hSD_{16}}\ar[u] & \Sigma^{7}E^{hSD_{16}}\ar[u]
}
\end{equation*}

The maximal finite subgroup of $\GG_2$ of $3$-power order is a cyclic group $C_3$ of order $3$;
it is unique up to conjugation in $\GG_2^1$.
The group $G_{24}$ is the maximal finite subgroup of $\GG_2$ containing $C_3$. The subgroup $SD_{16}$ is
the semidihedral group of order $16$. Because of the symmetry of this resolution, and because $\GG_2^1$ is
a virtual Poincar\'e duality group of dimension $3$, this is called a {\it duality resolution}. This resolution
and related resolutions were instrumental in exploring the $K(2)$-local category at primes $p > 2$.
See \cite{GHMV1}, \cite{HKM}, \cite{BehrensQ}, \cite{GHMCSC}, \cite{GHMRPic}, \cite{bcdual}, and \cite{Lader}. 
The latter paper makes a thorough exploration of what happens at $p > 3$. 

The main theorem of this paper provides an analog of the duality resolution at $n=2$ and $p=2$. The
prime $2$ is much harder for a number of reasons. First, the maximal finite $2$-subgroup of $\GG_2$ is
not cyclic, but isomorphic to $Q_8$, the quaternion group of order $8$. Second, every finite subgroup
of $\GG_2$ that we consider contains the central $C_2 = \{\pm 1\}$ of $\GG_2$; therefore, the homotopy groups
of the relevant fixed point spectra $E^{hF}$ are much more complicated. This means that the strategy of proof
of \cite{GHMR} won't work and we need to find another way. 

We now state our main result. As our chosen formal group $\Gamma_2$ we will use the formal group
of a supersingular elliptic curve over $\FF_4$. The curve will be defined over $\FF_2$, so that
$\GG_2 \cong \SS_2 \rtimes \Gal(\FF_4/\FF_2)$ where $\SS_2 = \Aut(\Gamma_2/\FF_4)$ is
the group of automorphisms of $\Gamma_2$ over $\FF_4$. Once again we have fiber sequence
\[
L_{K(2)}S^0 \longr E^{h\mathbb{G}_2^1}  \xrightarrow{\psi-1} E^{h\mathbb{G}_2^1}.
\]
We have $E^{h\mathbb{G}_2^1} = (E^{h\mathbb{S}_2^1})^{h\Gal}$ where $\Gal = \Gal(\FF_4/\FF_2)$
and $\SS_2^1 = \SS_2 \cap \GG_2^1$. For many
computational applications, the difference between $E^{h\GG_2^1}$ 
and $E^{h\SS_2^1}$ is innocuous. See Lemma \ref{more-split}.
Then our main result is this:

\noindent\begin{thm}\label{main} There exists a resolution of $E^{h\SS_2^1}$ in the $K(2)$-local category at the prime 2
$$
E^{h\mathbb{S}_2^1} \to E^{hG_{24}} \to E^{hC_6} \to \Sigma^{48} E^{hC_6} \to \Sigma^{48}E^{hG_{24}}\ .
$$
\end{thm}

The spectrum $E^{hC_6}$ is $48$-periodic, so $E^{hC_6}\simeq \Sigma^{48}E^{hC_6}$; this suspension is there
only to emphasize the symmetry in the resolution. 

Once again resolution means each successive composition is null-homotopic and all possible Toda brackets
are zero modulo indeterminacy; thus, the sequence  again refines to a tower of fibrations. 
This result has already had applications: it is an ingredient in Agn\`es Beaudry's analysis of the Chromatic
Splitting Conjecture at $p=n=2$. See \cite{beaudrychromatic}.

The apparent similarity of Theorem \ref{main}  with the prime $3$ analog is tantalizing, especially the 
suspension factor on the last term, but we as yet have no conceptual explanation. The proof at the prime
$2$ can be adapted to the prime $3$ but in both cases it comes down to a very specific, prime dependent
calculation.

A very satisfying feature of chromatic height $2$ is the connection with the theory of elliptic curves. The subgroup
$G_{24} \subseteq \SS_2$ is the automorphism group of our chosen supersingular curve inside the automorphisms 
of its formal group. Appealing to Strickland \cite{Strick} we can use this to get formulas for the action of 
$G_{24}$ on $E_\ast$, a necessary beginning to group cohomology calculations. Furthermore, we know that
$$
E^{hG_{48}} \simeq (E^{hG_{24}})^{h\Gal} \simeq L_{K(2)}Tmf
$$
where $Tmf$ is the global sections of the sheaf of $E_\infty$-ring spectra on the compactified stack
of generalized elliptic curves provided by Hopkins and Miller. See \cite{tmfbook}. Similarly, $E^{hC_6}$ is
the localization of global sections of the similar sheaf for elliptic curves with a level $3$ structure. See \cite{MR}. We
won't need anything like the full power of the Hopkins-Miller theory here, although we do use some of the 
calculations that arise from this point of view. See Section \ref{ch:coho}. 

We will prove Theorem \ref{main} in three steps.

First we prove, in Section \ref{ch:Tower}, that there exists a resolution
$$
E^{h\mathbb{S}_2^1} \to E^{hG_{24}} \to E^{hC_6} \to E^{hC_6} \to X
$$
where
$$
E_\ast X \cong E_\ast E^{hG_{24}}
$$
as twisted $\GG_2$-modules. 
This resolution and the necessary algebraic preliminaries were announced in \cite{HennRes}, and grew
out of the work surrounding \cite{GHMR}. The algebraic preliminaries are discussed in detail in
\cite{beaudry}. 

Second, in section \ref{ch:homotopy} we examine the Adams-Novikov Spectral Sequence
$$
H^\ast (G_{24},E_\ast) \Longrightarrow \pi_\ast X.
$$
Using a comparison of our resolution with a second resolution, due
to Henn, we show, roughly, that certain classes $\Delta^{8k+2} \in H^0(G_{24},E_{192k + 48})$ are permanent
cycles---which would certainly be necessary if our main result is true. The exact result is
in Corollary \ref{describer}. These calculations 
were among the main results in the first author's thesis \cite{thesis} and the key
ideas for the entire project can be found there. This ratifies a comment of Mark Mahowald that
there is a class in $\pi_{45}L_{K(2)}S^0$ which supports non-zero multiplications
by $\eta$ and $\kappabar$ and the only way this could happen is if $\Delta^2$ is a permanent 
cycle. This insight was, we think, the result of long hours of contemplation of the results of
Shimomura and Wang \cite{Shimat2} and, indeed, one reason for this entire project is to find some way to 
catch up with those amazing calculations. We will have more to say about this element
of Mahowald's in Remark \ref{whatisthisclass}.

Third and finally,  in section \ref{ch:mapping} we use a variation of this same comparison argument to produce
the equivalence $\Sigma^{48}E^{hG_{24}} \simeq X$. See Theorem \ref{maintheorem}. At the very 
end we add a remark about the possibility or impossibility of resolutions for $L_{K(2)}S^0$ itself. See
Remark \ref{notsphere}. 

There is a number of sections of preliminaries. Section \ref{ch:Knstuff} provides the usual background
on the $K(n)$-local category as well as some more specific information on mapping spaces. 
Section \ref{ch:coho} pulls together what we need about the homotopy groups of various fixed point spectra. This
draws from many sources, and we try to be complete there.

An essential ingredient in our argument is the existence of the other topological resolution for
$E^{h\mathbb{S}_2^1}$, Henn's centralizer resolution from \S 3.4 of \cite{HennRes}. We have
less control over the maps in this resolution, 
but, as has happened before (\cite{GHMV1}; see also \cite{bcdual}), this resolution provides essential
information needed to solve ambiguities in the duality resolution. It is also much closer to being an
Adams-Novikov-style resolution as it is based on relative homological algebra. We cover some of this
material at the end of section 3.

\subsection*{Acknowledgements} There is a wonderful group of people working on $K(2)$-local homotopy theory
and this paper is part of an on-going dialog within that community. We are grateful in particular to
Agn\`es Beaudry, Charles Rezk, and especially Hans-Werner Henn. Of course, others have helped as well.
In particular, the key idea---to look for
homotopy classes with particularly robust multiplicative properties---is due to Mark Mahowald. We miss 
his insight,  irreplaceable knowledge, and friendship every day. Finally, we would like to give a heartfelt
thanks to the referee, whose thorough and thoughtful report made this a better paper. 
\tableofcontents

\section{Recollections on the $K(n)$-local category}\label{ch:Knstuff}

We begin with the standard material on the $K(n)$-local category, the Morava stabilizer group,
and Morava $E$-theory, also known as Lubin-Tate theory. We then
get specific at $n=2$ and $p=2$, discussing the role of formal groups arising from supersingular
elliptic curves. We review some material on the homotopy type of the spectrum of maps between various fixed point 
spectra derived from Morava $E$-theory and then spell out some details of the $K(n)$-local
Adams-Novikov Spectral Sequence. Finally, we discuss, the role of the Galois group.

\subsection{The $K(n)$-local category} 

Fix a prime $p$ and let $n\geq 1.$ Let $\Gamma_n$ be a formal group of height $n$ over the finite field $\FF_p$ 
of $p$ elements. Then for any finite extension $i:\FF_p \subseteq \FF_q$ of $\FF_p$, we form the
group $\Aut(\Gamma_n/\FF_q)$ of the automorphisms of $i^\ast \Gamma_n$ over $\FF_q$. We fix
a choice of $\Gamma_n$ with the property that any extension $\FF_{p^n} \subseteq \FF_q$ 
gives an isomorphism
\begin{equation}\label{eq:stabilize-autos}
\Aut(\Gamma_n/\FF_{p^n}) \mathop{\longr}^{\cong} \Aut(\Gamma_n/\FF_{q}).
\end{equation}
The usual Honda formal group satisfies these criteria: this has a formal group law which is $p$-typical and with
$p$-series $[p](x)=x^{p^n}$. However, if $n=2$, then the formal group of a supersingular
elliptic curve defined over $\FF_p$ will also do, and this will be our preferred choice at $p=2$.
Define
\begin{equation}\label{morava-stab}
\SS_n = \Aut(\Gamma_n/\FF_{p^n}).
\end{equation}
If we choose a coordinate for $\Gamma_n$, then any element of $\SS_n$ defines a power series
$\phi(x) \in x\FF_{p^n}[[x]],$ invertible under composition, and the assignment $\phi(x) \mapsto \phi'(0)$ defines a map
$$
\SS_n \longr \FF_{p^n}^\times.
$$
This is a split surjection and we 
define $S_n$ to be the kernel of this map; this is the $p$-Sylow subgroup of the profinite group $\SS_n$.
We then get an isomorphism $S_n \rtimes \FF_{p^n}^\times \cong \SS_n$.

Define the (big) {\it Morava stabilizer group} $\GG_n$ as the automorphism group of the pair $(\FF_{p^n},\Ga_n)$.
Since $\Ga_n$ is defined over $\FF_p$, there is an isomorphism
\begin{equation}\label{big-morava-stab}
\GG_n \cong \Aut(\Gamma/\FF_{p^n}) \rtimes \Gal(\FF_{p^n}/\FF_p) = \SS_n \rtimes \Gal(\FF_{p^n}/\FF_p).
\end{equation}
We will often write $\Gal = \Gal(\FF_{p^n}/\FF_p)$ when the field extension is understood.

We next must define Morava $K$-theory. There are many variants, all of which have the same Bousfield
class and define the same localization; we will choose a variant which works well with Morava $E$-theory.
Let $K(n) = K(\FF_{p^n},\Ga_n)$ be the $2$-periodic ring spectrum with homotopy groups
$$
K(n)_\ast = \FF_{p^n}[u^{\pm 1}]
$$
and with associated formal group $\Ga_n$. Here the class $u$ is in degree $-2$. The group
$F_0 = \FF_{p^n}^\times \rtimes \Gal(\FF_{p^n}/\FF_p)$ acts on $K(n)$ and 
$$
\FF_p[v_n^{\pm 1}] \cong (K(n)^{hF_0})_\ast = K(n)^{F_0}_\ast
$$
where $v_n = u^{-(p^n-1)}$. The spectrum $K(n)^{hF_0}$ is thus a more classical version of Morava
$K$-theory.

We will spend a great deal of time working in the $K(n)$-local category and, when doing so, all our
spectra will implicitly be localized. In particular, we emphasize that we often write
$X \wedge Y$ for $L_{K(n)}(X \wedge Y)$, as this is the smash product
internal to the  $K(n)$-local category. 

We now define the Morava spectrum $E=E_n=E(\FF_{p^n},\Ga_n)$. (We will suppress the
$n$ on $E_n$ whenever possible to help ease the notation.) This is the complex
oriented, Landweber exact, $2$-periodic, $E_\infty$-ring spectrum with
\begin{equation}\label{lubin-tate}
E_\ast = (E_n)_\ast \cong \WW[[u_1,\cdots,u_{n-1}]][u^{\pm1 }]
\end{equation}
with $u_i$ in degree $0$ and $u$ in degree $-2$. Here $\WW = W(\FF_{p^n})$ is the ring of Witt
vectors on $\FF_{p^n}$. Note that
$E_0$ is a complete local ring with residue field $\FF_{p^n}$;
the formal group over $E_0$ is a choice of universal deformation
of the formal group $\Ga_n$ over $\FF_{p^n}$. (We will be specific about this choice
at $n=p=2$ below in subsection
\ref{definecurve}.) The group $\GG_n =    \Aut(\Ga_n) \rtimes \Gal(\FF_{p^n}/\FF_p)$
acts on $E=E_n$, by the Hopkins-Miller theorem \cite{GH1} and we have (see Section 1.5) a spectral sequence
for any closed subgroup $F \subseteq \GG_n$,
\begin{equation}\label{ANSS}
H^s(F,E_t) \Longrightarrow \pi_{t-s}E^{hF}.
\end{equation}
We will collectively call these by the name Adams-Novikov Spectral Sequence. See also Lemma
\ref{reduce-to-finite} below. If $F=\GG_n$
itself, then $E^{h\GG_n}\simeq L_{K(n)}S^0$ and we are computing the homotopy groups of
the $K(n)$-local sphere. 

Various subgroups of $\GG_n$ will play a role in this paper, especially at $n=2$ and $p=2$.
The right action of $\Aut(\Ga_n)$ on $\End(\Ga_n)$  defines a determinant map
$\det\colon \SS_n=\Aut(\Ga_n/\FF_{p^n}) \to \ZZ_p^\times$ which extends to a determinant map
\begin{equation}\label{det-defined}
\xymatrix@C=35pt{
\GG_n \cong  \SS_n \rtimes \Gal(\FF_{p^n}/\FF_p)
\rto^-{ \det \times 1} &  \ZZ_{p}^\times \times \Gal(\FF_{p^n}/\FF_p)
\rto^-{p_1} & \ZZ_p^\times.
}
\end{equation}
Define the {\it reduced determinant} (or {\it reduced norm}) $N$ to be the composition
\begin{equation}\label{norm-defined}
\xymatrix{
\GG_n \ar@/_1pc/[rr]_-N \rto^{\det} & \ZZ_p^\times \rto & \ZZ_p^\times/C \cong \ZZ_p.
}
\end{equation}
where $C \subseteq \ZZ_p^\times$ is the maximal finite subgroup. For example, $C = \{\pm 1\}$ if $p=2$.
There are isomorphisms $\ZZ_p^\times/C \cong \ZZ_p$ and we choose one.
Write $\GG_n^1$ for the kernel of $N$, $\SS_n^1 = \SS_n \cap \GG_n^1$, and
$S_n^1 = S_n \cap \GG_n^1$. The map $N:S_n \to \ZZ_p$ is split surjective and we have
semi-direct product decompositions for each of the groups $\GG_n$, $\SS_n$, and $S_n$; for example,
there is an isomorphism
$$
\SS_n^1 \rtimes \ZZ_p \cong \SS_n.
$$
If $n$ is prime to $p$, we can choose a central splitting and the semi-direct product is actually
a product, but that is not the case of interest here.

\subsection{Deformations from elliptic curves}\label{definecurve}
Here we spell out what we need from the theory of elliptic curves at $p=2$; this will give us
a preferred formal group and a preferred universal deformation. Choose $\Gamma_2$ to be the formal group 
obtained from the elliptic curve $C_0$ over $\FF_2$ defined by the Weierstrass equation
\begin{equation}\label{ss-curve}
y^2+y=x^3.
\end{equation}
This is a standard representative for the unique isomorphism class of supersingular curves
over $\overline{\FF}_2$; see \cite{Silverman}, Appendix A. Because $C_0$ is supersingular,
$\Gamma_2$ has height
$2$, as the notation indicates. Following \cite{Strick} let $C$ be the elliptic curve over 
$\WW(\mathbb{F}_4)[[u_1]]$ defined by the Weierstrass equation
\begin{equation}\label{ss-curve-def}
y^2+3u_1xy + (u_1^3-1)y=x^3.
\end{equation}
This reduces to $C_0$ modulo the maximal ideal $\mm = (2,u_1)$; the formal group $G$ of $C$ is a choice of 
the universal deformation of $\Gamma_2$. 

Again turning to \cite{Silverman}, Appendix A, we have
\begin{equation}\label{ss-curve1}
G_{24} \defeq \Aut(C_0/\FF_4) \cong Q_8 \rtimes \FF_4^\times
\end{equation}
where $\FF_4^\times \cong C_3$ acts on $Q_8$ as the $3$-Sylow subgroup of $\Aut(Q_8)$. Define
\begin{equation}\label{ss-curve2}
G_{48} \defeq \Aut(\FF_4,C_0) \cong \Aut(C_0/\FF_4) \rtimes \Gal(\FF_4/\FF_2).
\end{equation}
Since any automorphism of the pair $(\FF_4,C_0)$ induces an automorphism of  the
pair $(\FF_4,\Ga_2)$ we get a map $G_{48} \to \GG_2$. This map is an injection 
and we identify $G_{48}$ with its image. 

\begin{rem}\label{rem:levelstructures} Let  $C_0[3]$ be the subgroup scheme of $C_0$ consisting of points of order 
$3$;  over  $\FF_4$, this becomes abstractly isomorphic to $\Z/3 \times \Z/3$.  The group $G_{48}$ acts 
linearly on $C_0[3]$ and choosing  a basis for the $\FF_4$-points of $C_0[3]$ determines an isomorphism 
$G_{48} \cong GL_2(\Z/3)$. This restricts to an isomorphism group $G_{24} \cong Sl_2(\Z/3)$.
\end{rem}

\begin{rem}\label{the-subs} The following subgroups will play an important role in this
paper.
\begin{enumerate}

\item $C_2 = \{\pm 1\} \subseteq Q_8$;

\item $C_6 = C_2 \times \FF_4^\times$;

\item $C_4$, any of the subgroups of order $4$ in $Q_8$;

\item $G_{24}$ and $G_{48}$ themselves.
\end{enumerate} 
The subgroup $C_4$ is not unique, but it is unique up to conjugation in $G_{24}$ and in $\mathbb{G}_{2}$.
In particular, the homotopy type of $E^{hC_4}$ is well-defined.
\end{rem}

\begin{rem}\label{all-the-splittings} We have been discussing $G_{24}$ as a subgroup of $\SS_2$,
but it can also be thought of as a quotient. 
Inside of $\SS_2$ there is a normal torsion-free pro-$2$-subgroup $K$ which has the property that the composition 
$$
G_{24} \longr \SS_2 \longr \SS_2/K
$$
is an isomorphism. Thus we have a decomposition $K \rtimes G_{24} \cong \SS_2$.  See \cite{beaudry} for details. The group 
$K$ is a Poincar\'e duality group of dimension $4$.
\end{rem}

\def\mm{{\goth m}}

\subsection{The functor $E_\ast X$} We define
$$
E_\ast X = \pi_\ast L_{K(n)}(E \wedge X).
$$
Despite the notation, $E_\ast(-)$ is not a homology theory, as it does not take arbitrary wedges to sums,
but it is our most sensitive algebraic invariant on the $K(n)$-local category. 

Here are some properties of $E_\ast (-)$. See \cite{HvStr}, \S 8 and Appendix A for more
details. Let $\mm \subseteq E_0$ be the
maximal ideal and  $L = L_0$ be the 0th derived functor of completion at $\mm$. Recall that an $E_0$-module is
$L$-complete if the natural map $M \to L(M)$ is
an isomorphism. If $M = L(N)$, then $M$ is $L$-complete; that is, $L(N) \to L^2(N)$ is an isomorphism
for all $N$. Thus the full subcategory of $L$-complete modules is a reflexive sub-category of all 
continuous $E_0$-modules.
By Proposition 8.4 of \cite{HvStr}, $E_\ast X$ is $L$-complete for all $X$.

If $N$ is an $E_0$-module, there is a short exact sequence
\begin{equation}\label{eq:L-exact-seq}
0 \to \lim^1_k\, \Tor^{E_0}_1(E_0/m^k,N) \to L(N) \to \lim_k\, N/m^kN \to 0\, .
\end{equation}
Hence if $M$ is $L$-complete, then $M$ is $\mm$-complete, but it will be complete and separated only if the
right map is an isomorphism. See \cite{GHMR} \S 2 for some precise assumptions which guarantee that
$E_\ast X$  is complete and separated. All of the spectra in this paper will meet these assumptions.

Since $\GG_n$ acts on $E$, it acts on $E_\ast X$ in the category of $L$-complete modules. This action is twisted because it 
is compatible with the action of $\GG_n$ on the coefficient ring
$E_\ast$.  We will call $L$-complete $E_0$-modules with such a $\GG_n$-action either
{\it twisted $\GG_n$-modules}, or {\it Morava modules.}

For example, let $F \subseteq \GG_n$ be a closed subgroup. Then there is an isomorphism of twisted $\GG_n$-modules
\begin{equation}\label{eq:EstarEFbis}
E_\ast E^{hF} \cong \map(\GG_n/F,E_\ast)
\end{equation} 
where $\map(-,-)$ denotes the set of continuous maps. On the right hand side of this equation,
$E_\ast$ acts on the target and the $\GG_n$-action is diagonal. This needs a bit of care, as the 
proof of this result given in \cite{DH} is given for the Honda formal group and it may not be true
in general.  However, in analyzing the proof of the crucial Theorem 2 of \cite{DH} we see that
the isomorphism of (\ref{eq:EstarEFbis}) follows from the more basic isomorphism
$$
E_0E \cong \map(\GG_n,E_0)\,.
$$
The standard proof of this isomorphism for the Honda formal group
(see Theorem 12 of \cite{StrickGrossHop}, for example) requires only that our formal group be defined
over $\FF_p$ and satisfy the stabilization requirement of (\ref{eq:stabilize-autos}). 

\begin{rem}\label{vanishing-lines-1} We add here that Theorem 8.9 of \cite{HvStr} implies that the functor
$$
X \mapsto E_\ast X = \pi_\ast L_{K(n)}(E_n \wedge X)
$$
detects weak equivalences in the $K(n)$-local category. Given a map $f:X \to Y$, the class of spectra
$Z$ so that $L_{K(n)}(Z \wedge f)$ is a weak equivalence is closed under cofibrations and retracts; hence
if $E$ is in this class, then $L_{K(n)}S^0$ is in this class.
\end{rem}

\subsection{Mapping spectra}
 
We collect here some basics about the mapping spectra $F(E^{hF_1},E^{hF_2})$ for various subgroups $F_1$
and $F_2$ of $\GG_n$. 

\begin{rem}\label{EstarEF} Let $F \subseteq \GG_n$ be a closed subgroup. 
We begin with the  equivalence
\begin{equation}\label{theusual}
E \wedge E^{hF} \simeq \map(\GG/F,E)
\end{equation}
from the local smash product to  the localized spectrum of continuous maps. It is helpful to visualize this map as 
sending $x \wedge y$ to the
function $gF \mapsto x(gy)$. We will continue this mnemonic below: using point-wise defined
functions to indicate maps of spectra which cannot be defined that way. We hope the readers
can fill in the details themselves; if not, complete details can be found in \cite{GHMR} \S 2.

The action of $\GG_n$ on $E$ in (\ref{theusual}) defines the Morava module structure of
$E_\ast E^{hF}$; under the isomorphism of (\ref{theusual})
this maps to the diagonal action on the functions 
$$
(h\phi)(g) = h\phi(h^{-1}g).
$$
Note that
$$
(E_\ast E^{hF})^{\GG_n} = \map_{\GG_n}(\GG_n/F,E_\ast) \cong (E_\ast)^F\ .
$$
By \cite{DH}, Theorem 2, this extends to an isomorphism
$$
H^\ast(\GG_n,E_\ast E^{hF}) \cong H^\ast(F,E_\ast).
$$
See also Lemma \ref{reduce-to-finite} below.
\end{rem}

\begin{rem}\label{ghmres} We now recall some results from \cite{GHMR}.  
If $X = \lim\ X_i$
is a profinite set, let $E[[X]] = \lim\ E\wedge X_i^+$ where the $+$ indicates a disjoint
basepoint. Then if $F_1$ is a closed subgroup of $\GG_n$ we have an equivalence
\begin{equation}\label{mappingf}
E[[\GG_n/F_1]] \simeq F(E^{hF_1},E) 
\end{equation}
defined as follows. Let $F_E$ denote the function spectrum in $E$-modules. Then
\begin{align*}
E[[\GG_n/F_1]] &\simeq F_E(\map(\GG_n/F_1,E),E)\\
&\simeq F_E(E\wedge E^{hF_1},E)\\
&\simeq F(E^{hF_1},E).
\end{align*}
Next note that the equivalence of (\ref{mappingf}) is $\GG_n$-equivariant with the following
actions: in $F(E^{hF_1},E)$ we act on the target and in $E[[\GG_n/F_1]]$ we act
as follows:
$$
h(\sum\ a_g gF_1) = \sum h(a_g) h^{-1}gF_1.
$$

We can now make the following deductions. First suppose $F_1=U$ is open (and hence
closed), so that $\GG_n/F_1=\GG_n/U$ is finite. Let $F_2$ be finite. Then we have equivalences
\begin{align}\label{map-space-6}
\prod_{F_2\backslash \GG_n/U} E^{hF_x} &\simeq E[[\GG_n/U]]^{hF_2}\\
&\simeq F(E^{hU},E)^{hF_2}\nonumber\\
&\simeq F(E^{hU},E^{hF_2}).\nonumber
\end{align}
The product in the source is over the  double coset space, and for a double coset $F_2xU$,
$$
F_x = F_2 \cap xUx^{-1} \subseteq F_2.
$$
Since it depends on a choice of $x$, the group $F_x$ is defined only up to conjugation, but the fixed point spectrum 
$E^{hF_x}$ is well-defined up to weak equivalence. 

The first map of (\ref{map-space-6}) sends $a \in E^{hF_x}$ to the sum
$$
\sum_{g F_x \in F_2/F_x} (g^{-1}a)\ gxU.
$$

We say a word about the naturality of the equivalence of (\ref{map-space-6}). Suppose $U \subseteq V \subseteq 
\GG_n$ is a nested pair of open subgroups. Then for each double coset $F_2xU$ we get a double coset
$F_2xV$, a nested pair of subgroups
$$
F_x = F_2 \cap xUx^{-1} \subseteq F_2 \cap xVx^{-1}  = G_x\ ,
$$
and a transfer map
$$
\mathrm{tr}_x:E^{hF_x} \longr E^{hG_x}
$$
associated to this inclusion. Then we have a commutative diagram
\begin{equation}\label{tmap-space-60}
\xymatrix{
\prod_{F_2\backslash \GG_n/U} E^{hF_x}\dto_{\mathrm{tr}} \rto^{\simeq} &E[[\GG_n/U]]^{hF_2}\dto^g\\
\prod_{F_2\backslash \GG_n/V} E^{hG_x}  \rto^{\simeq} & E[[\GG_n/V]]^{hF_2}
}
\end{equation}
where the map $\mathrm{tr}$ is the sum of the transfer maps and the map $g$ induced by the quotient
$\GG_n/U \to \GG_n/V$.

For a more general closed subgroup $F_1$ write $F_1 = \cap_i U_i$ where
$U_i \subseteq \GG_n$ is open. Then, for $F_2$ finite we get a weak equivalence
\begin{equation}\label{thegencase}
\holim_i\ \prod_{F_2\backslash \GG_n/U_i} E^{hF_x} \simeq  F(E^{hF_1},E^{hF_2})
\end{equation}
where the product is as before and 
$$
F_x =F_2 \cap xU_ix^{-1} \subseteq F_2
$$
depends on $i$ and, as in (\ref{tmap-space-60}), there may be transfer maps  in the transition maps
for the homotopy limit. As $F_2$ is finite and the product in (\ref{thegencase}) is finite, this implies
that  the image of
$$
\prod_{F_2\backslash \GG_n/U_j} \pi_\ast E^{hF_x} \longr \prod_{F_2\backslash \GG_n/U_i} \pi_\ast E^{hF_x}
$$
is independent of $j$ for large $j$. It follows that there will be no $\lim^1$ term for the homotopy
groups of the inverse limit and hence there is an isomorphism
\begin{equation}\label{thegencase-homotopy}
\lim_i\ \prod_{F_2\backslash \GG_n/U_i} \pi_\ast E^{hF_x} \cong \pi_\ast F(E^{hF_1},E^{hF_2})
\end{equation}
\end{rem}

\subsection{The $K(n)$-local Adams-Novikov Spectral Sequence} This is the main technical tool
of this paper, and we give a few details of the construction and some of its properties. We begin 
with some algebra from \cite{HvStr}, Appendix A. 

Let $\mm \subset E_0 \cong \WW[[u_1,\cdots,u_{n-1}]]$ be the maximal ideal. An $L$-complete
$E_0$-module $M$ is {\it pro-free} if any one of the following equivalent conditions is satisfied:
\begin{enumerate}

\item $M \cong L(N) \cong N^{\wedge}_\mm$ for some free $E_0$-module $N$;

\item the sequence $(p,u_1,\cdots,u_{n-1})$ is regular on $M$;

\item $\mathrm{Tor}_1^{E_0}(\FF_{p^n},M) = 0$; 

\item $M$ is projective in the category of $L$-complete modules.
\end{enumerate} 

Proposition A.13 of \cite{HvStr} implies that a continuous homomorphism $M \to N$ of pro-free modules
is an isomorphism if and only if $M/\mm M \to N/\mm N$ is an isomorphism.

If $F \subseteq \GG_n$ is a closed subgroup, then $E_\ast E^{hF}$ is pro-free, by (\ref{eq:EstarEFbis}).

By Proposition 8.4 of \cite{HvStr}, if $X$ is a spectrum with
$K(n)_\ast X$ concentrated in even degrees, then $E_\ast X$ is pro-free, concentrated in even degrees.
Furthermore, if $E_\ast X$ is pro-free and concentrated in even degrees, then 
$$
K(n)_\ast X \cong \FF_{p^n} \otimes_{E_0} E_\ast X \cong (E_\ast X)/\mm.
$$

\begin{rem}\label{vanishing-lines}
We now come to the Adams-Novikov Spectral Sequence in $K(n)$-local category. As
in \cite{StrickGrossHop}, Proposition 15,  this spectral sequence is obtained by the standard cosimplicial
cobar complex for $E=E_n$ in the $K(n)$-local category. Thus we have
$$
E_2^{s,t} \cong \pi^s\pi_t(E^\bullet \wedge X) \Longrightarrow \pi_{t-s}L_{K(n)}X.
$$
The smash products are in the $K(n)$-local category. It is a consequence of the proof of
\cite{StrickGrossHop}, Proposition 15 that this spectral sequence converges strongly and has a horizontal
vanishing line at $E_\infty$. 
\end{rem}

We'd now like to rewrite the $E_2$-term as group cohomology, at least under some hypotheses. For any group
$G$, let $EG$ be the standard contractible simplicial $G$-set with $(EG)_s = G^{s+1}$. This is a bar construction. If
$G$ is a topological group, then $EG$ is a simplicial space.

Let  $M$ be an $L$-complete Morava module. We define
$$
H^s(\GG_n,M) = \pi^s\map_{\GG_n}(E\GG_n,M)
$$
where $\map_{\GG_n}(-,-)$ denotes the group of continuous $\GG_n$-maps. Since there is an
isomorphism
$\map_{\GG_n}(\GG_n^{s+1},M) \cong \map(\GG_n^{s},M)$, we have that $H^\ast(\GG_n,M)$
is the $s$th cohomology group of a cochain complex
$$
\xymatrix{
M \ar[r] & \map(\GG_n,M) \ar[r] & \map(\GG_n^{2},M) \ar[r] & \cdots
}
$$

\begin{prop}\label{e2-is-grp-cohio}  Suppose $X = Y \wedge Z$ where $K(n)_\ast Y$ is concentrated in 
even degrees and $Z$ is a finite complex. Then we have an isomorphism
$$
\pi^s\pi_t(E^\bullet \wedge X)  \cong H^s(\GG_n,E_t X)
$$
and the $K(n)$-local Adams-Novikov Spectral Sequence
reads
$$
H^s(\GG_n,E_t X) \Longrightarrow \pi_{t-s}L_{K(n)}X.
$$
\end{prop}

\begin{proof} If $Z = S^0$, then $E_\ast X$ is pro-free in even degrees. The result can be found in
Theorem 4.3 of \cite{BHANSS}. The authors there work with the version of Morava $E$-theory obtained from
the Honda formal group, but they need only the isomorphism
$E_0E \cong \map(\GG_n,E_0)$, which holds in our case. See the remarks after \eqref{eq:EstarEFbis}. The
key idea is that this last isomorphism can be extended to an isomorphism
$$
E_\ast (E^{\wedge s}\wedge X) \longr \map(\GG_n^s,E_\ast X)
$$
for any spectrum $X$ with $E_tX$ pro-free for $t$. 
For more general $Z$, the isomorphism on $E_2$-terms follows from the five lemma. 
\end{proof} 

In some crucial cases, it is possible to reduce the $E_2$-term to 
group cohomology over a finite group. Let $F \subseteq \GG_n$
be a finite subgroup and $N$ an $L$-complete twisted $F$-module. Define an $L$-complete module
$N\! \uparrow_{F}^{\GG_n}$ as the set of continuous maps $\phi:\GG_n \to N$ such that $\phi(hx) = h\phi(x)$ 
for $h \in F$. This becomes a Morava module with
$$
(g\phi)(x) = \phi(xg)
$$
with $g \in \GG_n$ and there is an isomorphism
$$
\map_{\GG_n}(\GG_n^{s+1},N\! \uparrow_{F}^{\GG_n}) \cong \map_{F}(\GG_n^{s+1},N)
$$
of groups of continuous maps. Let $H^s(F,N) = \pi^s\map_F(EF,N)$. 

\begin{lem}[{\bf Shapiro Lemma}]\label{lem:Shapiro}
Let $F \subseteq \GG_n$ be a finite subgroup and suppose $N$ is an $L$-complete 
twisted $F$-module with the property that $N/\mm N$ is finite dimensional over $\FF_{p^n}$. Then there is an 
isomorphism
$$
H^\ast(\GG_n,N\! \uparrow_{F}^{\GG_n}) \cong H^\ast(F,N)
$$
and, under these hypotheses on $N$, there is an isomorphism
$$
H^\ast(F,N) \cong \lim_k\ H^s(F,N/m^kN).
$$
\end{lem} 

\begin{proof}Since $N$ is $L$-complete and $N/\mm N$ is finite, the short exact sequence of (\ref{eq:L-exact-seq})
implies $N \cong N^\wedge_\mm$.

Choose a nested sequence $U_{i+1} \subseteq U_{i} \subseteq \GG_n$ of finite
index subgroups of $\GG_n$ with the property that $\cap\, U_i =\{e\}$.
Then for all $s \geq 0$ we have
\begin{align*}\map_{\GG_n}(\GG_n^{s+1},N\! \uparrow_{F}^{\GG_n}) &\cong
						\map_{F}(\GG_n^{s+1},N)\\
& \cong \lim_k\, \map_{F}(\GG_n^{s+1},N/\mm^kN)\\
& \cong \lim_k\, \colim_i\, \map_{F}((\GG_n/U_i)^{s+1},N/\mm^kN).\\
\end{align*}
The last isomorphism follows from the fact that $N/m^kN$ is discrete and finite. Since $F$ is finite, we have 
$F \cap U_i = \{e\}$ for all $i$ greater than some $i_0$. Then for $i > i_0$ it follows that 
$(\GG_n/U_i)^{\bullet +1}$ is a contractible simplicial free $F$-set and, thus,
we have an isomorphism
$$
\pi^s\map_{F}((\GG_n/U_i)^{\bullet+1},N/\mm^kN) \cong H^s(F,N/\mm^kN)\, .
$$
Again since $F$ is finite, $H^s(F,N/\mm^kN)$ is a finite abelian group. Both isomorphisms
of the lemma now follow.
\end{proof}

\begin{lem}\label{reduce-to-finite} Let $Y$ be a spectrum equipped with an isomorphism of Morava modules
$$
E_\ast Y \cong \map(\GG_n/F,E_\ast) \cong E_\ast E^{hF}
$$
where $F \subseteq \GG_n$ is a finite subgroup. Then for all finite spectra $Z$, there is an isomorphism
$$
H^\ast(\GG_n,E_\ast (Y \wedge Z)) \cong H^\ast (F,E_\ast Z).
$$
\end{lem}

\begin{proof} We have an isomorphism of Morava modules
$$
\xymatrix{
E_\ast (Y \wedge Z) \ar[r]^-\cong& \map(\GG_n/F,E_\ast Z)
}
$$
where the action of $\GG_n$ on the target is by conjugation: $(g\phi)(x) = g\phi(g^{-1}x)$. There is an isomorphism of
Morava modules
$$
\map(\GG_n/F,E_\ast Z) \cong  (E_\ast Z)\! \uparrow_{F}^{\GG_n}
$$
adjoint to evaluation at $eF$. 
The result follows from the Shapiro Lemma \ref{lem:Shapiro} .
\end{proof}

\begin{rem}\label{rem:why-did-we-bother}  Suppose $Y = E^{hF}$. Then Lemma \ref{reduce-to-finite} follows
from the variant of Theorem 2 of \cite {DH} appropriate for our formal group; indeed, we need
only require that $F$ be closed. 
However, the proof relies on the construction of $E^{hF}$ given in 
that paper and doesn't {\it a priori} apply if we only know $E_\ast Y \cong E_\ast E^{hF}$ -- as will be the 
case in our application.
\end{rem}

\begin{rem}\label{not-p-typical} There is a map of Adams-Novikov Spectral Sequences
$$
\xymatrix{
\Ext^s_{BP_\ast BP}(\Sigma^t BP_\ast, BP_\ast) \ar@{=>}[r] \ar[d] & \ZZ_{(2)} \otimes \pi_{t-s}S^0 \ar[d]\\
H^s(\GG_2,E_t) \ar@{=>} [r] & \pi_{t-s}L_{K(2)}S^0
}
$$
but it takes a little care to define. Let $G(x,y) \in E^0[[x,y]]$ be the formal group law of the supersingular curve of
(\ref{ss-curve-def}); since $G$ is the formal group of a Weierstrass curve, it has a preferred coordinate. 
Let $G_\ast(x,y) = uG(u^{-1}x,u^{-1}y) \in E^\ast[[x,y]]$. Then $G_\ast$ is a formal group over
$E^\ast$ with coordinate in cohomological degree $2$ and, therefore, is classified by
a map $\ZZ_{(2)} \otimes MU_\ast \to E_\ast$. Since $G_\ast$ is not evidently $2$-typical, it need not
be classified by a map $BP_\ast \to E_\ast$. However, over a $\ZZ_{(2)}$-algebra,
the Cartier idempotent  gives an equivalence between the groupoid of all formal group laws and
the groupoid of $2$-typical formal group laws; hence we have a diagram of spectral sequences as needed:
$$
\xymatrix{
\Ext^s_{BP_\ast BP}(\Sigma^t BP_\ast, BP_\ast) \ar@{=>}[r]&  \ZZ_{(2)} \otimes \pi_{t-s}S^0\\
\ZZ_{(2)}\otimes \Ext^s_{MU_\ast MU}(\Sigma^t MU_\ast, MU_\ast) \ar@{=>}[r] \ar[u]^-\cong \ar[d] &
 \ZZ_{(2)} \otimes \pi_{t-s}S^0\ar[u]_= \ar[d]\\
H^s(\GG_2,E_t) \ar@{=>}[r]& \pi_{t-s}L_{K(2)}S^0 \ .
}
$$
We will use this below in the section on the cohomology of $G_{48}$. 
\end{rem} 

\subsection{The action of the Galois group}
We now turn to analyzing $E^{h\SS_2}$ as an equivariant spectrum over the Galois group. As above,
we will write $\Gal = \Gal(\FF_{p^n}/\FF_p)$, so that $\GG_n \cong \SS_n \rtimes \Gal$.

We begin with the following elementary fact: the map $\ZZ_p \to \WW$ is Galois with
Galois group $\Gal$; thus, it is faithfully flat, \'etale, and
the shearing map
$$
\WW \otimes_{\ZZ_p} \WW \to \map(\Gal,\WW)$$
sending $a \otimes b$ to the function $g \mapsto ag(b)$ is an isomorphism. In fact, this shearing map
is certainly an isomorphism modulo $p$; then the statement for $\WW$ follows from Nakayama's Lemma.
Faithfully flat descent now implies that
the category of $\ZZ_p$-modules is equivalent to the category of twisted $\WW[\Gal]$-modules under 
the functor $M \mapsto \WW \otimes_{\ZZ_p} M$; the inverse to this functor sends $N$ to $N^{\Gal}$.

This extends to the following result.

\begin{lem}\label{coh-decomp-galois} Let $K \subseteq \GG_n$ be a closed subgroup and let
$K_0 = K \cap \SS_n$. Suppose the canonical map
$$
K/K_0 \longr \GG_n/\SS_n \cong \Gal
$$
is an isomorphism. Then for any twisted $\GG_n$-module $M$ we have isomorphisms
\begin{align*}
H^\ast(K,M)  &\cong H^\ast(K_0,M)^{\Gal}\\
H^\ast(K_0,M) &\cong \WW \otimes_{\ZZ_p} H^\ast(K,M)\ .
\end{align*}
\end{lem}

\begin{proof} The subgroup $\SS_n$ acts on $E_0$ through $\WW$-algebra homomorphisms;
hence it acts on $M$ through $\WW$-module homomorphisms. It follows that we can write
the functor $(-)^{K}$ of invariants as a composite functor
$$
\xymatrix{
\hbox{twisted $\GG_n$-modules} \rto^-{(-)^{K_0}} & \hbox{twisted $\WW[\Gal]$-modules} \rto^-{(-)^{\Gal}} &
\hbox{$\ZZ_p$-modules}.
}
$$
As we just remarked, the second 
of these two functors is an equivalence of categories and  in particular it is an exact functor.
The first equation follows. The second equation  follows from the first and the fact that the inverse to $(-)^{\Gal}$ is
the functor $M \mapsto \WW \otimes_{\ZZ_p}M$.
\end{proof}

We now give a fact seemingly known to everyone, but hard to find in print. Drew Heard was the first
to point out an error in our original argument; others followed quickly. We learned the
following replacement from Mike Hopkins, Agn\`es Beaudry, and the referee. We extend our thanks to everyone. 

\begin{lem}\label{hzero} For all $p$ and all $n \geq 1$ we have isomorphisms
\begin{align*}
H^0(\GG_n,E_0) &\cong \ZZ_p\\
H^0(\SS_n,E_0) &\cong  \WW = W(\FF_{p^n}).
\end{align*}
Furthermore, $H^0(\GG_n, E_t) = H^0(\SS_n,E_t) = 0$ if $t \ne 0$.
\end{lem}

\begin{proof} By Lemma \ref{coh-decomp-galois} we need only do the case of $\SS_n$.

It is also sufficient to prove this when $\Gamma_n$ is the Honda formal group over $\FF_{p^n}$.
Any other height $n$ formal group becomes isomorphic to the Honda formal group over the algebraic closure
of $\FF_p$, and the general result could then be deduced from Galois descent. 

Write $\phi$ for lift of Frobenius to Witt vectors $\WW$.
For the Honda formal group, $\SS_n$ is the group of units in the endomorphism ring
$$
\End(\Gamma_n) \cong \WW\langle S\rangle/(S^n-p)
$$
where $\WW\langle S\rangle$ is non-commutative polynomial ring over $\WW$ on a variable $S$ with $Sa = \phi(a)S$
when $a \in \WW$.
Thus we have an inclusion $\WW^\times \subset \SS_n$.
Every $k\in\ZZ_p^\times \subseteq \SS_n$ corresponds to an automorphism $\Gamma_n \to [k]_{\Gamma_n}$ of $
\Gamma_n$; in particular $\ZZ_p^\times$ acts trivially on $E_0$ and if $u \in E_{-2}$
is the invertible generator, we have $k_\ast u =ku$.
If follows that $H^0(\ZZ_p^\times ,E_t) = 0$ if $t \ne 0$.

This leaves the case $t=0$. Write $\KK = \WW[p^{-1}]$. Since all elements of $\SS_n$ fix the constants
$\WW \subseteq E_0$ there is an inclusion
$$
\WW \cong H^0(\WW^\times,\WW) \to H^0(\WW^\times,E_0)\ .
$$
Furthermore, since $E_0$ is torsion-free,
it is sufficient to show $H^0(\WW^\times,E_0[p^{-1}]) \cong \KK$. By \cite{DHaction}, Lemma 4.3, there are power series
$w_i = w_i(u_1,u_2,\ldots,u_{n-1}) \in \KK[[u_1,\cdots,u_{n-1}]]$ and an inclusion
$$
\WW[[w_1,\cdots,w_{n-1}]] \longr \KK[[u_1,\cdots,u_{n-1}]]
$$
onto an $\GG_n$-equivariant sub-algebra which becomes an isomorphism after inverting $p$ in the source. Thus
it is sufficient to show that 
$$
H^0(\WW^\times,J) = 0
$$
where $J = \WW[[w_1,\cdots,w_{n-1}]]/\WW$. 

By \cite{DHaction}, Proposition 3.3, the action of $\WW^\times$ on $\WW[[w_1,\cdots,w_{n-1}]]$ is diagonal:
if $a \in \WW^\times$, then
$$
a_\ast w_i = \phi^i(a)a^{-1}w_i.
$$
Let $\omega \in \WW$ be a primitive $(p^n-1)$st root of unity and set $a = 1 + p\omega$. Then 
$\phi(1+p\omega) = 1 +p \omega^p$. If some $i_j \ne 0$ it follows that   
$$
(1 + p\omega)_\ast w_1^{i_1}\ldots w_{n-1}^{i_{n-1}} \ne w_1^{i_1}\ldots w_{n-1}^{i_{n-1}}
$$
as needed. 
\end{proof}

\begin{rem}\label{hzero-strong} Notice that the proof of Lemma \ref{hzero} actually shows
we need only relatively small subgroups of $\GG_n$ to get the full invariants. Specifically, we have
$$
\ZZ_p \cong H^0(\WW^\times \rtimes \Gal,E_0)
$$
and if $t \ne 0$, then $H^0(\ZZ_p^\times,E_t)=0$. We won't need this stronger result. 
\end{rem}

\begin{rem}\label{etalediffs} In the proof of Lemma \ref{decompose} below we will use the following observation. 
Suppose
we have a spectral sequence $\{E_r^{s,t}\}$  that is multiplicative; that is, $E_r^{\ast,\ast}$ is a bigraded 
ring which is commutative up to sign and $d_r$ satisfies the Leibniz rule, again up to sign.
Further suppose $R \subseteq E_r^{0,0}$ is a commutative
subring of $d_r$-cycles and $R \subseteq S$ is an \'etale extension in $E_r^{0,0}$. Then every
element of $S$ is a $d_r$-cycle. To see this, note that $d_r$ restricted to $S$
is a derivation over $R$ and any such derivation must vanish; indeed, depending on your foundations, 
the vanishing of such derivations may even be part of your definition of \'etale.
\end{rem}

\begin{lem}\label{decompose} For all $p$ and all $n \geq 1$ there is 
a $\Gal$-equivariant equivalence
$$
\Gal^+  \wedge L_{K(n )}S^0 \to E^{h\SS_n}.
$$
\end{lem}

\begin{proof} We first prove we have injection $\WW \to \pi_0E^{h\SS_n}$. We begin with the
isomorphism $\WW \cong H^0(\SS_n,E_0)$ of Lemma \ref{hzero}. Since
$\ZZ_p \subseteq \WW$ is an \'etale extension and the Adams-Novikov Spectral Sequence for $\SS_n$ is a
spectral sequence of rings, all of $\WW$ survives to $E_\infty$ and the edge homomorphism
provides a surjection $\pi_0E^{h\SS_n} \to \WW$ of rings.  The kernel of this map
nilpotent as an ideal because the spectral sequence has a horizontal vanishing line at
$E_\infty$. See Remark \ref{vanishing-lines}. We then have a diagram
$$
\xymatrix{
\ZZ_p \ar[r] \ar[d]_{\subseteq} & \pi_0E^{h\SS_n} \ar[d]\\
\WW \ar[r]_{=} \ar@{-->}[ur]& \WW
}
$$
Again since $\ZZ_p \subseteq \WW$ is \'etale, the dashed arrow can be completed uniquely.
This yields the injection we need.

Now define  $\omega:S^0 \to E^{h\SS_n}$ to be a representative of the homotopy class defined by
a primitive $(p^n-1)$st root of unity in $\WW$. We can extend  $\omega$ to a $\Gal$-equivariant map
$f:\Gal^+ \wedge S^0 \to E^{h\SS_2}$ inducing the splitting $\WW \to \pi_0E^{h\SS_2}$;
here we use that the map $\ZZ_p[\Gal] \to \WW$
$$
\sum a_g g \longmapsto \sum a_g g(\omega)
$$
is an isomorphism. The map $f$ 
extends to an isomorphism of twisted $\GG_n$-modules 
$$
E_\ast f: E_\ast (\Gal^+\wedge S^0) \cong \map(\Gal,E_\ast) \cong \map(\GG_n/\SS_n,E_\ast) \cong
E_\ast E^{h\SS_n}\ 
$$
thus completing the argument.
\end{proof}

The following result is a topological analog of Lemma \ref{coh-decomp-galois}.

\begin{lem}\label{more-split} Let $K \subseteq \GG_n$ be a closed subgroup and let $K_0 = K \cap \SS_n$.
Suppose the canonical map
$$
K/K_0 \longr \GG_n/\SS_n \cong \Gal
$$
is an isomorphism. Then there is a $\Gal$-equivariant equivalence
$$
\Gal^+ \wedge E^{hK} \to E^{hK_0}.
$$
\end{lem}

\begin{proof} This follows from Lemma \ref{decompose}. Define a map
$E^{h\SS_n} \wedge E^{hK} \to E^{hK_0}$ by the composition
\begin{equation}\label{ivelosttrack}
E^{h\SS_n} \wedge E^{hK} \to E^{hK_0} \wedge E^{hK_0} \to E^{hK_0}
\end{equation}
where the first map is given by the inclusion and the last map is multiplication.
We have
\begin{align*}
E_\ast(E^{h\SS_n} \wedge E^{hK}) &\cong \map(\GG_n/\SS_n,E_\ast) \otimes_{E_\ast} \map(\GG_n/K,E_\ast)\\
& \cong \map(\GG_n/\SS_n \times \GG_n/K,E_\ast)\ . 
\end{align*}
In $E_\ast$-homology, the map of (\ref{ivelosttrack}) then becomes the map  
\begin{align*}
\map(\GG_n/\SS_n,E_\ast) \otimes_{E_\ast} \map(\GG_n/K,E_\ast) \cong
\map(\GG_n/\SS_n &\times \GG_n/K,E_\ast)\\ 
 &\longr \map(\GG_n/K_0,E_\ast)
\end{align*}
induced by the maps on cosets
$$
\GG_n/K_0 \to \GG_n/K_0 \times \GG_n/K_0 \to \GG_n/\SS_n \times \GG_n/K
$$
where the first map is  the diagonal and the second map is projection. Since $K/K_0 \cong \GG_n/\SS_n$,
this map on cosets is an isomorphism; therefore, the map of (\ref{ivelosttrack}) is an $E_\ast$-isomorphism.
By Remark \ref{vanishing-lines-1} this map is a weak equivalence.
\end{proof}

\begin{rem}\label{coh-implies} Combining Lemma \ref{coh-decomp-galois} and
Lemma \ref{more-split} yields an isomorphism of spectral sequences
$$
\xymatrix{
\WW \otimes H^\ast(K,E_\ast) \ar@{=>}[r] \dto_\cong &
\WW \otimes \pi_\ast E^{hK}\dto^\cong \\
H^\ast(K_0,E_\ast) \ar@{=>}[r]  & \pi_\ast E^{hK_0}
}
$$
where the differentials on the top line are the $\WW$-linear differentials extended
from the spectral sequence for $K$.
\end{rem}

\begin{rem}\label{Goodfor48} At $n=2$ and $p=2$, Lemma \ref{more-split}
applies to the case of $K=G_{48}$; then $K_0 = G_{24}$.
This implies that any of the spectra
$$
X(i,j) \defeq \Sigma^{24i} E^{hG_{48}} \vee \Sigma^{24j} E^{hG_{48}}
$$
has the property that $E_\ast X(i,j) \cong E_\ast E^{hG_{24}}$. But $X(i,j) = \Sigma^{24i}E^{hG_{24}}$
if and only if $i \equiv j$ mod $8$. This means that in some of the arguments we give to prove
our main result  we will have to produce two homotopy classes rather than one. See Theorem
\ref{maintheorem}.
\end{rem}

\section{The homotopy groups of homotopy fixed point spectra}\label{ch:coho}

Here we collect what we will need about the homotopy groups of $E^{hF}$, where $F$ runs through
the finite subgroups of $\mathbb{G}_2^1$ of Remark \ref{the-subs}. We will be working entirely at $n=p=2$
and using the formal group from the supersingular curve of (\ref{ss-curve}). Much of what's needed is in the literature
and we'll do our best to give references. However, much of what is written is for calculations over
Hopf algebroids, which is not quite what we're doing, and the results need translation. In addition, 
many of the results as written include some variant of the phrase
``we neglect the $bo$-patterns''.  We make this thought precise
with the following {\it ad hoc} definition. 

\begin{defn}\label{bo-patterns} Let $F \subseteq \GG_2$ be any finite subgroup containing
$C_2=\{\pm 1\}$. Then we define the {\bf $bo$-patterns} $L_1(\pi_\ast E^{hF})$
of $\pi_\ast E^{hF}$ to be the image of the map in homotopy
$$
\pi_\ast E^{hF} \longr \pi_\ast L_{K(1)}E^{hF}.
$$
We also define the {\bf pure $K(2)$-classes} $M_2(\pi_\ast E^{hF})$ to be the kernel of
the same map.
\end{defn}

Thus we have a short exact sequence
$$
0 \to M_2(\pi_\ast E^{hF}) \to \pi_\ast E^{hF} \to L_1(\pi_\ast E^{hF}) \to 0.
$$
Notice that the $bo$-patterns are defined as a quotient. In most cases, this sequence is not split
as modules over the homotopy groups of spheres.

\begin{rem}\label{boisko} The name $bo$-patterns is something of a misnomer, as $KO$-patterns
would be more accurate. Here $KO$ is $8$-periodic $2$-complete real $K$-theory. In all our
examples we will have an isomorphism
$$
R(F) \otimes_{\mathbb{Z}_2} KO_\ast \cong \pi_\ast L_{K(1)}E^{hF}
$$
for some $\ZZ_2$-algebra $R(F)$ in degree zero.  While $R(F) \otimes_{\mathbb{Z}_2} KO_\ast$ is $8$-periodic, $L_1(\pi_\ast E^{hF})$
will typically have $8k$-periodicity for some $k>1$. As a warning, we mention that this isomorphism
is simply as rings; we are not claiming $L_{K(1)}E^{hF}$ is a $KO$-algebra.
\end{rem}

\begin{rem}\label{whatwerewethinking} Here is more detail, to explain our thinking.

Let us write $S/2^n$ for the $\ZZ/2^n$-Moore spectrum. Then there
is a weak equivalence
$$
L_{K(1)}X \simeq \holim\ v_1^{-1}(X \wedge S/2^n)
$$
and, if $X$ is $K(2)$-local, a corresponding localized Adams-Novikov Spectral Sequence
\begin{equation}\label{localv1ss}
\lim\ v_1^{-1}H^\ast (\GG_2,(E_\ast X)/2^n) \Longrightarrow \pi_\ast L_{K(1)}X.
\end{equation}
This spectral sequence doesn't obviously converge.

Now suppose $C_2 = \{\pm 1\} \subseteq F \subset G_{48}$. Then the spectral sequence (\ref{localv1ss}) for
$X = E^{hF}$ becomes
\begin{equation}\label{localv1ssF}
\lim\ v_1^{-1}H^\ast (F,(E_\ast)/2^n) \Longrightarrow \pi_\ast L_{K(1)}E^{hF}.
\end{equation}
Using Strickland's formulas \cite{Strick} it is possible to show that
$$
\lim\ v_1^{-1}H^\ast (F,(E_\ast)/2^n) \cong \lim\ v_1^{-1}H^\ast (C_2,(E_\ast)/2^n)^{F/C_2}
$$
and that
$$
v_1^{-1}H^\ast (C_2,(E_\ast)/2^n) \cong \WW((u_1))[u^{\pm 2},\eta]
$$
where $\WW((u_1)) = \lim (\WW/2^n)[[u_1^{\pm 1}]]$ and 
$\eta \in H^1(C_2,E_2)$ detects the class of the same name in $\pi_1S^0$. (See (\ref{cohc2}) below.)
From this it follows
that the spectral sequence (\ref{localv1ssF}) is completely determined by the standard differential
$d_3(v_1^2) = \epsilon \eta^3$, where $v_1^2 = u_1^2u^{-2}$ and $\epsilon \in \FF_4((u_1))^\times$ is
a unit. We can conclude that the spectral sequence converges and
$$
\pi_\ast L_{K(1)}E^{hF} \cong \WW((u_1))^{F}\otimes_{\ZZ_2} KO_\ast
$$
and in particular, that $\pi_\ast L_{K(1)}E^{hF}$ and $L_1(\pi_\ast E^{hF})$ are both concentrated in
degrees congruent to $0$, $1$, $2$, and $4$ modulo $8$. It then remains to analyze the pure
$K(2)$-local classes. 

Now, not much of what we just wrote is explicitly in print,  and it would 
take quite a few pages to prove in detail. But we will put together what we can from the existing literature
to cover the main points case-by-case below. See Propositions \ref{coh-C2-v1}, \ref{coh-c4-v1}, and
\ref{coh-G48-v1}
\end{rem}

\subsection{The homotopy groups of $E^{hC_2}$ and $E^{hC_6}$}\label{homotopyC2}

The standard source here is Mahowald-Rezk \cite{MR}, which uses the Hopf algebroid approach, but
also uses the same elliptic curve we have chosen. So the translation is straightforward. Here
is a summary.

The central $C_2\subseteq \GG_2$ acts trivially on $E_0$ and by multiplication by $-1$ on 
$u$; hence
\begin{equation}\label{cohc2}
H^\ast(C_2,E_\ast) \cong \WW[[u_1]][u^{\pm 2},\alpha]/(2\alpha)
\end{equation}
where $\alpha \in H^1(C_2,E_2)$ is the image of the generator of $H^1(C_2,\ZZ\langle\mathrm{sgn}\rangle)$
under the map which sends the generator of the sign representation to $u^{-1}$. Since 
$$
v_1 = u_1u^{-1} \in H^0(C_2,E_2/2)
$$
and the class $u_1\alpha \in H^1(C_2,E_2)$ is the image of $v_1$ under the integral Bockstein, the
class $\eta \in \pi_1S^0$ is detected by $u_1\alpha$. We will also write $\eta = u_1\alpha$. 

\begin{prop}\label{coh-C2-v1} The class $b \defeq u_1^2u^{-2}$ reduces to $v_1^2$ in
$v_1^{-1}H^\ast(C_2,E_\ast/2)$. There is an isomorphism
$$
\WW((u_1))[b^{\pm 1}, \eta]/(2\eta) \cong \lim\ v_1^{-1}H^\ast(C_2,E_\ast/2^n).
$$
\end{prop}

The standard differential $d_3(v_1^2) = \eta^3$ (see Lemma \ref{muhelps} below) forces
a differential
$$
d_3(u^{-2}) = \epsilon u_1\alpha^3
$$
where $\epsilon \in \FF_2[[u_1]]^\times$. Using the Mahowald-Rezk transfer argument \cite[Prop. 3.5]{MR} we have
$\nu \in \pi_3S^0$ is non-zero in $\pi_3E^{hC_2}$ and detected by $\alpha^3$; this in turn
forces a differential
$$
d_7(u^{-4}) = \alpha^7 = \alpha \nu^2.
$$
The spectral sequence collapses at $E_7$ and we have the following result.

\begin{prop}\label{homotopyc2} The homotopy ring $\pi_\ast E^{hC_2}$ is periodic of period 16 with
periodicity generator $e_{16}$ detected by $u^{-8}$. The $bo$-patterns $L_1(E^{hC_2})$ are concentrated in
degrees congruent to  $0$, $1$, $2$, and $4$ modulo 8 and the group of pure $K(2)$-local classes
$M_2(E^{hC_2})$
is generated by the classes
$$
\alpha^{i}e^k_{16},\qquad k \in \ZZ,\  i=3,4,5,6.
$$
\end{prop}

To get the homotopy of $E^{hC_6}$, with $C_6 = C_2 \times \FF_4^\times$,
we need to know the action of $\FF_4^\times$. We can use 
Strickland's calculations \cite{Strick} or interpret the Mahowald-Rezk results. 
Let $\omega \in \FF_4^\times$ be a primitive cube root of unity. Then, $\omega_\ast u=
\omega u$ and $\omega_\ast u_1 = \omega u_1$; it follows that $\omega_\ast \alpha = \omega^{-1} \alpha$.
The next result can be deduced from these formulas and the fact that
$$
\pi_\ast E^{hC_6} \cong (\pi_\ast E^{hC_2})^{\FF_4^\times}. 
$$

\begin{prop}\label{homotopyc6} The homotopy ring $\pi_\ast E^{hC_6}$ is periodic of period 48 with
periodicity generator $e_{16}^3$ detected by $u^{-24}$. The $bo$-patterns $L_1(E^{hC_6})$ are concentrated in
degrees congruent to  $0$, $1$, $2$, and $4$ modulo 8 and the group of pure $K(2)$-local classes
$M_2(E^{hC_6})$ is generated by the classes
$$
e_{16}^{3k}\alpha^3\qquad e_{16}^{3k}\alpha^6\qquad e_{16}^{3k+1}\alpha^4\qquad e_{16}^{3k+2}\alpha^5
$$
of degrees $48k+3$, $48k+6$, $48k+20$, and $48k+37$ respectively.
Furthermore, the homotopy class $\nu \in \pi_3S^0$ is detected by the class $\alpha^3$ and the
class $\kappabar \in \pi_{20}S^0$ is detected by $e_{16}\alpha^4$.
\end{prop}

\subsection{The homotopy groups of $E^{hC_4}$}

The standard reference for this calculation is Behrens-Ormsby \cite{BO} \S 2.1, especially
Theorem 2.1.3 and Perspective 2, after Remark 2.1.9. See also their Figure 2.
Again they use Hopf algebroids. Here we hit another small problem: the
supersingular elliptic curve they use is different from the one we have chosen, and thus they have
a different action of $C_4$ on a different version of Morava $E$-theory. There are two possible solutions.
One is to do the calculations over again, using Strickland's formulas. The other is to notice that the two supersingular
curves become isomorphic over the algebraically  closed field $\overline{\FF}_2$ and to use descent to make the
calculations. Using either method we obtain the following result. Let $i \in C_4$ be a generator and
let 
$$
z = u_1+ i_\ast u_1 \in H^0(C_4,E_0).
$$
There are further cohomology classes 
$$
b_2 \in H^0(C_4,E_4)\qquad \delta \in H^0(C_4,E_8)
$$
and
$$
\gamma \in H^1(C_4,E_6)\qquad \xi \in H^2(C_4,E_8).
$$
Let $\eta \in H^1(C_4,E_2)$ and $\nu \in H^1(C_4,E_4)$ be the images of the like-named classes from
the $BP$-based Adams-Novikov Spectral Sequence (see Remark \ref{not-p-typical}).

\begin{prop}\label{coh-c4} There is an isomorphism
$$
H^\ast (C_4,E_\ast) \cong \WW[[z]][b_2,\delta^{\pm 1},\eta,\nu,\gamma, \xi]/R
$$
where $R$ is the ideal of relations given by
$$
2\eta = 2\gamma = 4\xi = 0
$$
and
$$
b_2^2 \equiv  z^2\delta \qquad \mathrm{mod}\ 2
$$
and 
$$
\delta \eta^2  = b_2\xi = \gamma^2 \qquad b_2 \gamma = z \delta\eta
$$
and
$$
b_2\eta = z\gamma\qquad \gamma\eta = z\xi 
$$
and the final relations involving $\nu$:
$$
\nu^2 = 2\xi\qquad  2\nu=z\nu=\eta\nu = b_2\nu=\gamma\nu= 0
$$
\end{prop}

\begin{proof} 
This can be obtained from Perspective 2 (after Remark 2.1.9) of \cite{BO} by a three-step
process. First set $\tilde{\gamma}=\gamma$, $\tilde{j}=z-2$, and $\beta=\delta^{-1}\xi$. Second,
invert $\delta$. Finally, complete at the maximal ideal of $H^0(C_4,E_0)$.
\end{proof}

This result is displayed in Figure 1 below, presented in the standard Adams format: the $x$-axis
is $t-s$; the $y$-axis is $s$. In this chart, the square box $\Box$ represents
a copy of $\WW[[z]]$, the circle $\circ$ a copy of $\FF_4[[z]]$, and the crossed circle
$\otimes$ a copy of $\WW[[z]]/(4,2z)$ generated by a class of the form $\xi^j\delta^i$.
A solid dot is a copy of $\FF_4$ annihilated by $z$; it is generated by a class of the form $\xi^j\nu$.
The solid lines are multiplication by
$\eta$ or $\nu$, as needed, and a dashed line indicates that $x\eta = z y$, where $x$ and $y$
are generators in the appropriate bidegree.

\begin{figure}[h]\label{cohc4-fig}
\centering
\includegraphics{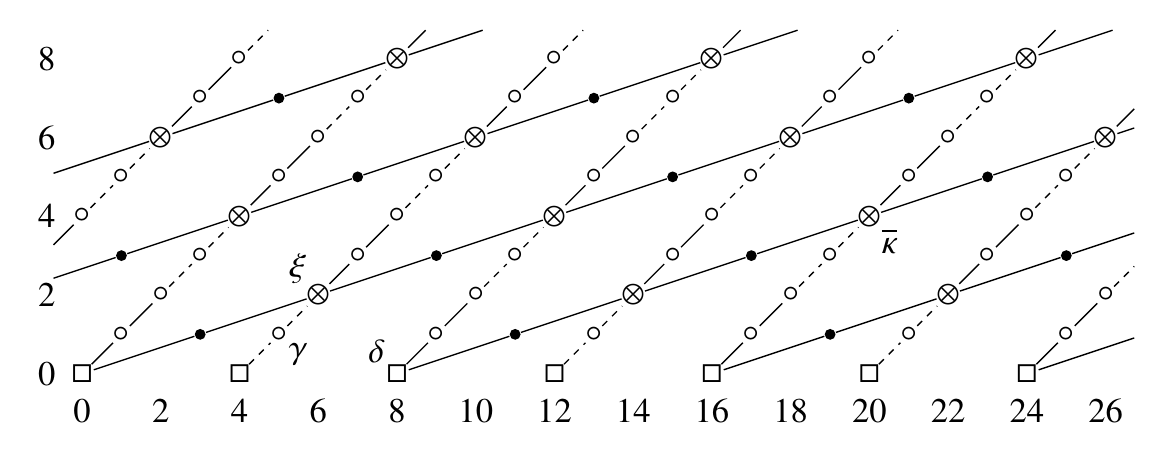}
\caption{The cohomology of $C_4$}
\end{figure}

Notice that $H^\ast (C_4,E_\ast)$ is $8$-periodic with the periodicity class $\delta$ and that 
$$
\times \xi: H^s(C_4,E_\ast) \to H^{s+2}(C_4,E_{\ast + 8})
$$
is onto for $s \geq 0$ and an isomorphism
for $s > 0$. In fact $\delta^{-1}\xi \in H^2(C_4,E_0)$ is, up to multiplication by a unit, the image
of the periodicity class for the group cohomology of $C_4$ under the inclusion of trivial coefficients:
$$
\ZZ/4\cong H^4(C_4,\ZZ_2) \cong H^2(C_4,\WW) \to H^s(C_4,E_0).
$$

\begin{prop}\label{coh-c4-v1} Modulo $2$ we have an equivalence $b_2 \equiv v_1^2$ and then
an isomorphism
$$
\WW((z))[b_2^{\pm 1}, \eta]/(2\eta) \cong \lim\ v_1^{-1}H^\ast(C_4,E_\ast/2^n).
$$
\end{prop}

The differentials and extensions in this spectral sequence go exactly as in Behrens and Ormsby
\cite{BO}, Theorem 2.3.12. We end with the following result.

\begin{prop}\label{homotopyc4} The homotopy ring $\pi_\ast E^{hC_4}$ is periodic of period 32 with
periodicity generator $e_{32}$ detected by $\delta^4$. The $bo$-patterns $L_1(E^{hC_4})$ are concentrated in
degrees congruent to  $0$, $1$, $2$, and $4$ modulo 8 and the group of pure $K(2)$-local classes 
$M_2(E^{hC_4})$
is generated by the classes $e_{32}^kx$ where $x$ is from the following table
\begin{center}
\begin{tabular}{|c|c|c|c|} 
\hline
Class& Degree & Order & $E_2$-name \\ 
\hline
a&$1$&$2$&$\delta^{-1}\xi\nu$\\
\hline
$a\eta $&$2$&$2$&$\delta^{-2}\xi^2\nu^2$\\
\hline
$\nu$ & $3$ &$4$ & $\nu$ \\ 
\hline
$a\nu $&$4$&$2$&$\delta^{-1}\xi \nu^2$\\
\hline
$\nu^2$ & $6$ & $2$ &$\nu^2$ \\ 
\hline
$\epsilon$ & $8$ & $2$& $\delta^{-2}\xi^4$ \\ 
\hline
$\nu^3=\eta\epsilon$ & $9$ & $2$& $\kappabar^2\delta^{-5}\xi\nu$ \\ 
\hline
$\kappa$ & $14$ & $2$ & $\delta\nu^2$ \\ 
\hline
$b$ & $19$ & $4$ & $\delta^2\nu$ \\ 
\hline
$\overline{\kappa}$ & $20$ & $4$ & $\delta\xi^2$ \\ 
\hline
$\eta\overline{\kappa}$ & $21$ & $2$ &$\eta\xi^3$ \\ 
\hline
$b\nu$ & $22$ & $4$ &$\delta^2\nu^2$ \\ 
\hline
$c$ & $27$ & $2$ & $\delta^2\gamma\xi$ \\ 
\hline
$c\eta $ & $28$ & $2$ & $\delta^{-1}\xi^6$ \\ 
\hline
\end{tabular}
\end{center}
\end{prop}

The pure $K(2)$ classes in $\pi_\ast E^{hC_4}$ are presented in the Figure 2: the horizontal bar is multiplication by
$\nu$, the diagonal bar is multiplication by $\eta$. Note that $\kappabar\eta^2 = 2b\nu$.
The horizontal scale is the degree of the element, but the vertical scale has no meaning. 
Many of the additive and multiplicative relations are given by exotic extensions in this
spectral sequence and the meaning of the original Adams-Novikov filtration becomes
attenuated as a result; see Behrens-Ormsby \cite{BO}, especially figure 9, for details.

\begin{figure}[h]
\centering
\includegraphics[width=\textwidth]{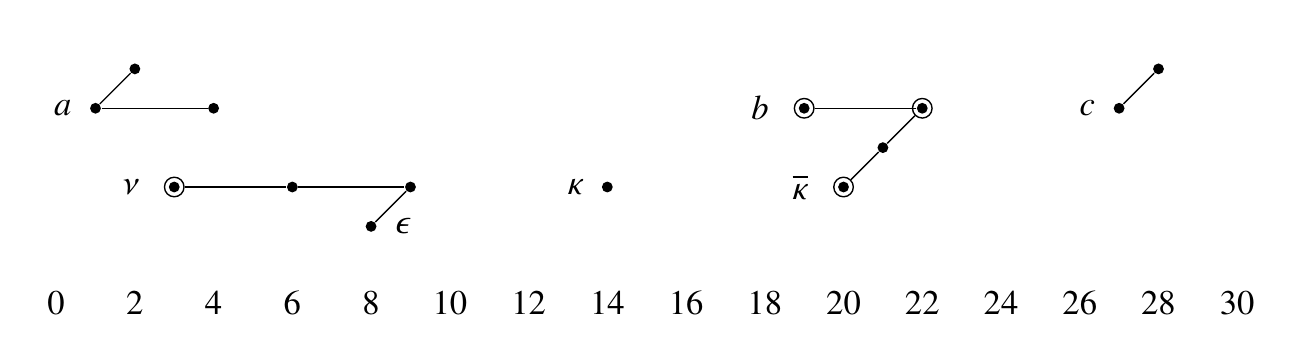}
\caption{The pure $K(2)$ classes in $\pi_\ast E^{hC_4}$}
\end{figure}

\subsection{The homotopy groups of $E^{hG_{24}}$ and $E^{hG_{48}}$}

Remarks \ref{coh-implies} and \ref{Goodfor48} yield an isomorphism of spectral sequences
$$
\xymatrix{
\WW \otimes_{\mathbb{Z}_2} H^\ast(G_{48},E_\ast) \ar@{=>}[r] \dto_\cong & \WW \otimes \pi_\ast E^{hG_{48}}\dto_\cong \\
H^\ast(G_{24},E_\ast) \ar@{=>}[r]  & \pi_\ast E^{hG_{24}}\ .
}
$$
Therefore, we focus on the case of $G_{48}$. Here the standard sources 
are \cite{Tilman}, \cite{tmfbook} and \cite{HM} although
it requires some translation in each case to get the results we want.

The ring $H^0(G_{48},E_\ast)$ is isomorphic to the ring of modular forms for supersingular elliptic curves at the 
prime $2$. Then there are elements
$$
c_4 \in H^0(G_{48},E_8)\qquad c_6 \in H^0(G_{48},E_{12}) \qquad \Delta \in H^0(G_{48},E_{24})
$$
obtained from the modular forms of the same name for our supersingular curve. Since this curve is
smooth, $\Delta$ is invertible and the $j$-invariant of our curves $j = c_4^3/\Delta \in H^0(G_{48}, E_0)$ is defined. Then we get
an isomorphism
$$
\ZZ_2[[j]][c_4,c_6,\Delta^{\pm 1}]/(c_4^3 - c_6^2 = (12)^3\Delta,\Delta j = c_4^3) \cong H^0(G_{48},E_\ast).
$$
Modulo $2$ we get a slightly simpler answer: 
$$
\FF_2[[j]][v_1,\Delta^{\pm 1}]/(j\Delta = v_1^{12}) \cong H^0(G_{48},E_\ast/2).
$$
Modulo $2$ we have congruences
\begin{equation}\label{relationtov1}
c_4 \equiv v_1^4\qquad\mathrm{and}\qquad c_6 \equiv v_1^6.
\end{equation}

To describe the higher cohomology, we make a table of multiplicative generators. For each $x$, the
bidegree of $x$ is $(s,t)$ if $x \in H^s(G_{48},E_t)$.  All but $\mu$
detect the elements of the same name in $\pi_\ast S^0$. Furthermore, all elements but $\kappabar$ are
in the image of the map (see Remark \ref{not-p-typical})
$$
\Ext^{\ast,\ast}_{BP_\ast BP}(BP_\ast,BP_\ast) \to
H^\ast(\GG_2,E_\ast) \to H^\ast (G_{48},E_\ast)\ .
$$
Hence we also give the name (the ``MRW'' is for Miller-Ravenel-Wilson) of a preimage.  The Greek letter
notation is that of \cite{MRW}.

\begin{center}
\begin{tabular}{|c|c|c|c|} 
\hline
Class& Bidegree & Order &MRW\\ 
\hline
$\eta$&$(1,2)$&$2$&$\alpha_1$\\
\hline
$\nu$&$(1,4)$&$4$&$\alpha_{2/2}$\\
\hline
$\mu$ & $(1,6)$ &$2$&$\alpha_3$\\ 
\hline
$\epsilon$ & $(2,10)$ & $2$&$\beta_2$\\ 
\hline
$\kappa$ & $(2,16)$ & $2$&$\beta_3$\\ 
\hline
$\overline{\kappa}$ & $(4,24)$ & $8$&$-$\\ 
\hline
\end{tabular}
\end{center}
The class $\kappabar \in \pi_{20}S^0$ is detected by the image of $\beta_4$ in $H^2(\GG_2,E_\ast)$. 
The class $\mu$ has a special role which we discuss below in Lemma \ref{muhelps}, but we would like to note right away that 
\begin{equation}\label{relationtov1bis}
 v_1^2\eta \equiv \mu\quad \mathrm{modulo}\ 2.
\end{equation}

The following result is actually much easier to visualize than to write down. See the Figure 3  below.

\begin{thm}\label{coh-G48} There is an isomorphism
$$
H^0(G_{48},E_\ast)[\eta,\nu,\mu,\epsilon,\kappa,\kappabar]/R 
\cong H^\ast(G_{48},E_\ast)
$$
where $R$ is the ideal defined by
\begin{enumerate}

\item the order of the elements of positive cohomological degree:
$$
2\eta = 4\nu = 2\mu = 2\epsilon = 2\kappa = 8\overline{\kappa} = 0;
$$

\item the relations for $\nu$: 
$$
\eta\nu = 2\nu^2=\nu^4=\mu\nu = 0;
$$

\item the  relations for $\epsilon$:
$$
\eta\epsilon = \nu^3,\quad \nu\epsilon=\epsilon^2=\mu\epsilon=0;
$$

\item the relations for $\kappa$:
$$
\nu^2\kappa = 4\overline{\kappa},\quad \eta^2\kappa=\epsilon\kappa = \kappa^2=\mu\kappa=0;
$$

\item the elements annihilated by modular forms
$$
c_4\nu = c_6\nu = c_4\epsilon = c_6\epsilon = c_4\kappa = c_6\kappa=0;
$$

\item the relations between $\overline{\kappa}$ and modular forms;
$$
c_4\overline{\kappa}=\Delta\eta^4,\quad c_6\overline{\kappa}=\Delta\eta^3\mu; 
$$

\item and the relations indicated by the congruences of (\ref{relationtov1}) and (\ref{relationtov1bis}):
$$
\mu^2 = c_4\eta^2 \qquad c_4\mu = c_6\eta \qquad c_6\mu =c_4^2 \eta \ .
$$

\end{enumerate}
\end{thm}

This result is presented graphically in Figure 3 below. We present it as the $E_2$ page of the Adams-Novikov 
Spectral Sequence. The cohomology is $24$-periodic on $\Delta,$ and the spectral sequence fills the entire upper-half plane.
In Figure 3, the square box $\Box$ represents
a copy of $\ZZ_2[[j]]$, the circle $\circ$ a copy of $\FF_2[[j]]$, and the crossed circle
$\otimes$ a copy of $\ZZ_2[[j]]/(8,2j)$ generated by a class of the form $\Delta^i\kappabar^j$.
The solid bullet represents a class of order $2$ annihilated by $j$
and the doubled bullet a class of order $4$ annihilated by $j$; these last classes
are always of the form $\Delta^i\kappabar^j\nu$.  The solid lines are 
multiplication by $\eta$ or $\nu$, as needed, and a dashed line indicates that $x\eta  = j y$, where $x$ and $y$
are generators in the appropriate bidegree.

\begin{figure}[h]
\centering
\includegraphics[width=\textwidth]{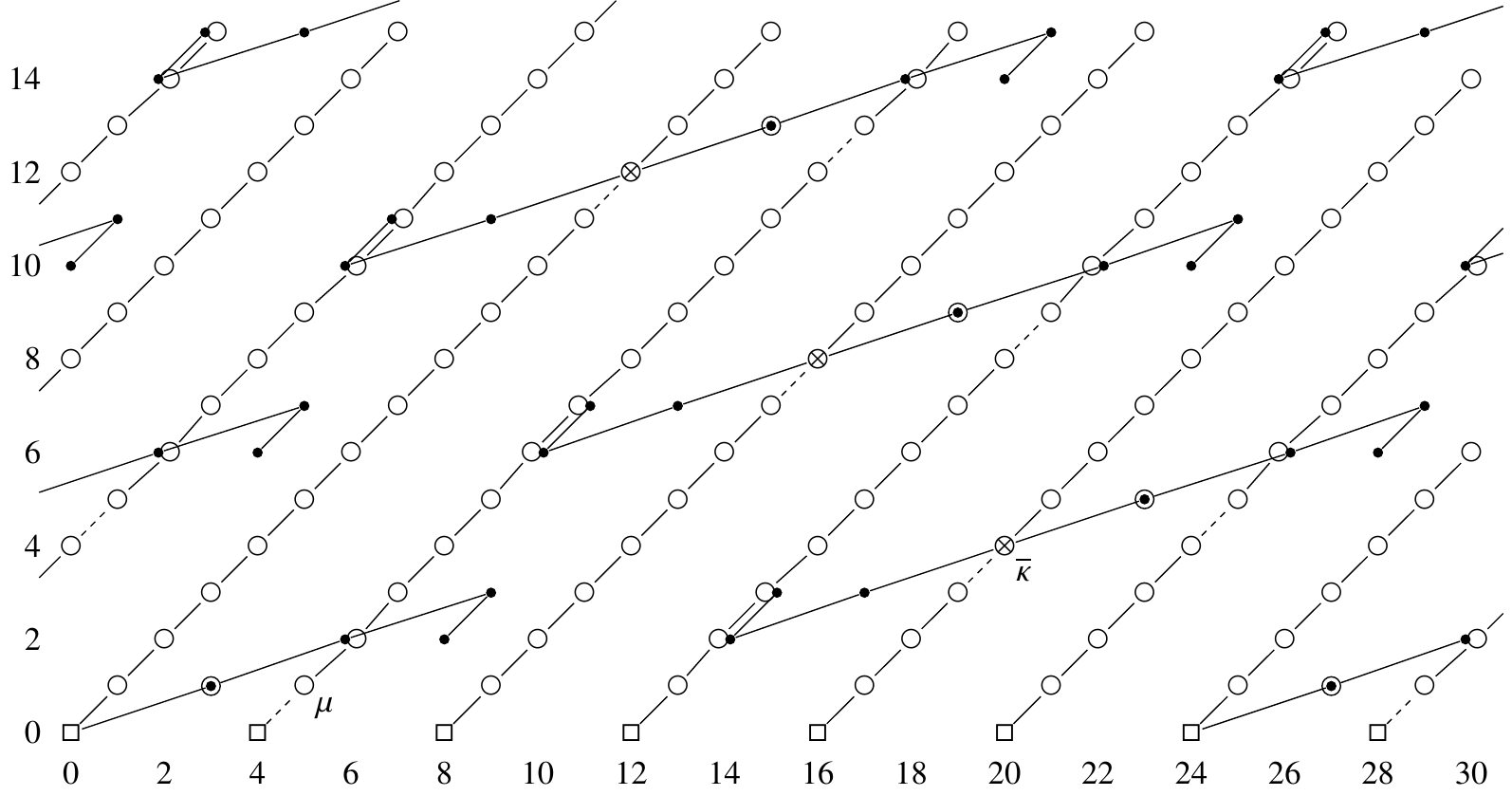}
\caption{The cohomology of $G_{48}$}
\end{figure}

\begin{rem}\label{whatsupwithj} (1) Many of the later relations can be rephrased as relations for multiplication
by $j=c_4^3/\Delta$. For example Theorem \ref{coh-G48} (4) implies
$$
j\nu=j\epsilon=j\kappa=0
$$
and (5) implies
$$
 j\overline{\kappa} = c_4^2\eta^4 
$$
and (6) implies
$$
j\mu = c_4^2c_6\Delta^{-1}\eta.
$$
These last two equations explain the dashed lines in Figure 4.

(2) Multiplication by $\kappabar:H^s(G_{48},E_t) \to H^{s+4}(G_{48},E_{t+24})$ is surjective and an
isomorphism if $s > 0$. In fact,  up to a unit,  $\Delta^{-1}\kappabar \in H^4(G_{48},E_0)$ is the image
of the periodicity class in group cohomology for $Q_8$ under the inclusion of trivial coefficients:
$$
\ZZ/8\cong H^4(Q_8,\WW)^{G_{48}/Q_8} \cong H^4(G_{48},\WW) \to H^4(G_{48},E_0).
$$
\end{rem}

The congruence (\ref{relationtov1}) and the relations of Theorem \ref{coh-G48} now give the 
following result. Note that the class $c_4$ becomes invertible in
$\lim\ v_1^{-1}H^\ast(G_{48},E_\ast/2^n)$ and we may define $b_2 = c_6/c_4$. ({\bf Warning:} This class $b_2$
is related to, but not quite the same, as the class $b_2$ of Proposition \ref{coh-c4}. Both uses of $b_2$
appear in the literature.) 

\begin{prop}\label{coh-G48-v1} The class $b_2$ reduces to $v_1^2$ in
$v_1^{-1}H^\ast(G_{48},E_\ast/2)$. There are isomorphisms
$$
\ZZ_2((j))[b_2^{\pm 1}, \eta]/(2\eta) \cong \lim\ v_1^{-1}H^\ast(G_{48},E_\ast/2^n).
$$
and
$$
\FF_2((j))[v_1^{\pm 1},\eta] \cong v_1^{-1}H^\ast(G_{48},E_\ast/2).
$$

Under the reduction map $H^\ast (G_{48},E_\ast) \to H^\ast(G_{48},E_\ast/2)$ we have
\begin{align*}
c_4 &\mapsto v_1^4\\
c_6 &\mapsto v_1^6\\
c_4 \kappabar &\mapsto v_1^4\kappabar= \Delta\eta^4 \\
\mu &\mapsto v_1^2\eta\ .
\end{align*}
Under the localization map $H^\ast (G_{48},E_\ast) \to v_1^{-1}H^\ast(G_{48},E_\ast/2)$ we have
\begin{align*}
\Delta &\mapsto v_1^{12}/j\\
\kappabar & \mapsto v_1^{8} \eta^4/j \\
\end{align*}
and that $\nu$, $\epsilon$, and $\kappa$ map to zero.
\end{prop}

We have the following; see \cite{HM}, \cite{Tilman}, or \cite{tmfbook}.

\begin{prop}\label{homotopyG48} The homotopy ring $\pi_\ast E^{hG_{48}}$ is periodic of period 192 with
periodicity generator detected by $\Delta^8$. The $bo$-patterns $L_1(E^{hG_{48}})$ are concentrated in
degrees congruent to  $0$, $1$, $2$, and $4$ modulo 8.
\end{prop}

\begin{rem}\label{chart48explain} We will not try to enumerate the pure $K(2)$-classes of $M_2(E^{hG_{48}})$; 
this information is known (by the same references as for Proposition \ref{homotopyG48}),
but we won't need that information in its entirety and it is rather complicated
to write down. 
What we will need can be read off of Figure 4, which is adapted from the charts created
by Tilman Bauer \cite{Tilman}, Section 8.

This chart shows a section of the $E_\infty$-page of the Adams-Novikov Spectral Sequence
$$
H^s(G_{48},E_t) \Longrightarrow \pi_{t-s}E^{hG_{48}}.
$$
It is in the standard Adams bigrading $(t-s,s)$.

\begin{figure}[h]
\begin{center}
 \includegraphics[page=1,width=\textwidth]{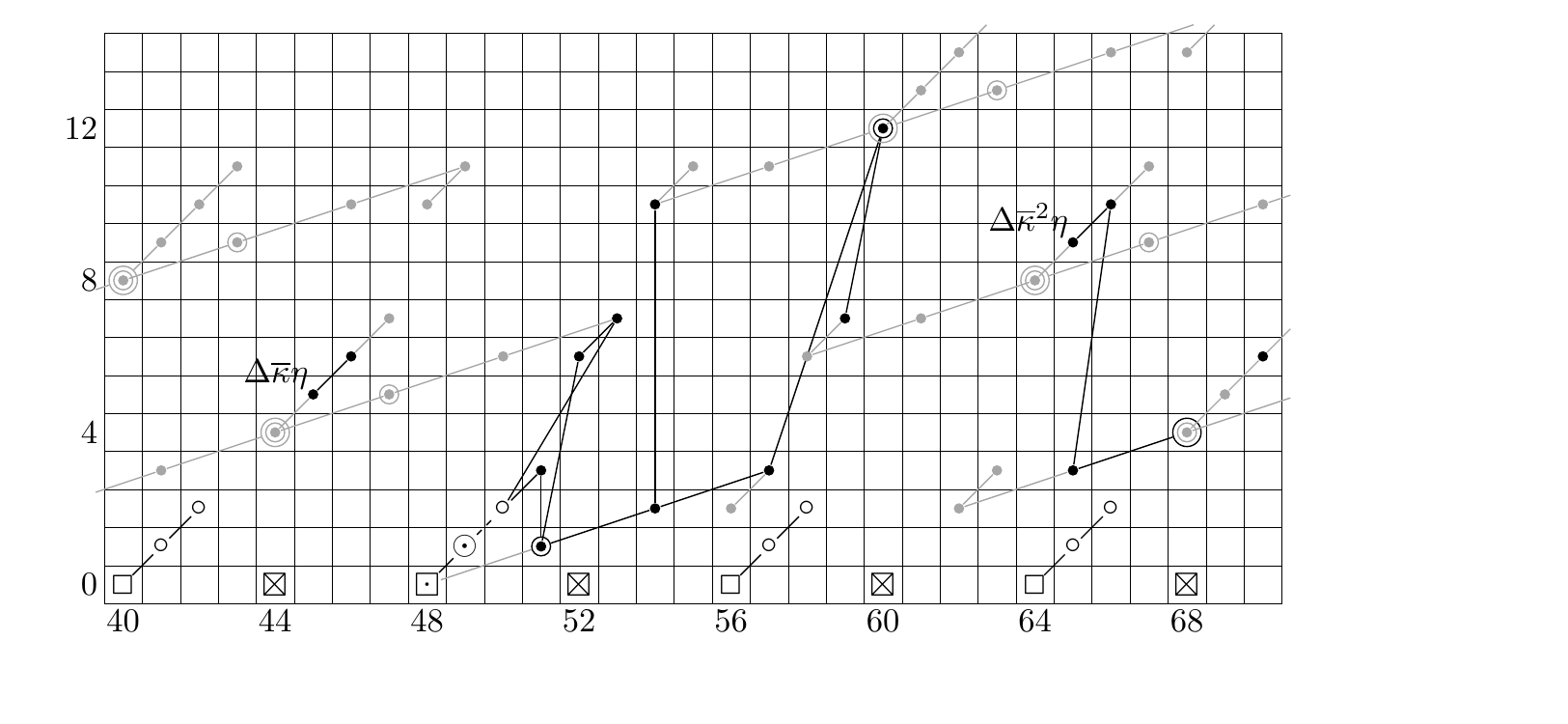}
 \end{center}
 \vspace{-1cm}
\caption{The  homotopy groups $\pi_iE^{hG_{48}}$ for $40 \leq i \leq 70$.}
 \end{figure}

Some additive and multiplicative extensions are displayed as well. 
The non-zero permanent cycles are in black; some other elements, mostly 
built from patterns around elements of the form $\Delta^j\kappabar^i$, have been left in gray for orientation,
even though they do not last to the $E_\infty$-page. Bullets with circles are elements of order $4$; bullets with two 
circles are elements of order $8$. Vertical lines are extensions by multiplication by $2$, lines raising homotopy 
degree by $1$ are $\eta$-extensions, lines raising homotopy degree by $3$ are $\nu$-extensions. 

The lines $0 \leq s \leq 2$ display the $bo$-patterns; the adorned boxes and circles all represent
ideals of either $\ZZ_2[[j]]$ or $\FF_2[[j]]$:
\begin{align*}
\Box &\cong \ZZ_2[[j]]\\
\boxtimes &\cong (2) \subseteq \ZZ_2[[j]]\\
\boxdot&\cong (4,j) \subseteq \ZZ_2[[j]]\\
\circ &\cong \FF_2[[j]]\\
\odot &\cong (j) \subseteq \FF_2[[j]].
\end{align*}
Elements not falling into one of these patterns are annihilated by $j$. The $\eta$-extension from $(t-s,s) = (65,3)$
entry is
ambiguous. We mark it as non-zero because we may choose, as Bauer does, the two generators of the group of
pure $K(2)$-classes in $\pi_{65}E^{hG_{48}}$ to be
$$
e[45,5]\kappabar  \qquad \mathrm{and} \qquad  e[51,1]\kappa
$$
where $e[45,5] \in \pi_{45}E^{hG_{48}}$ and $e[51,1] \in \pi_{51}E^{hG_{48}}$ are generators detected
by $\Delta\kappabar\eta$ and $\Delta^2\nu$ respectively. The class $e[51,1]\kappa$ is detected on the
$s=3$ line by $\Delta^2\kappa\nu$ and $e[51,1]\kappa\eta \ne 0$.

\end{rem}

We now record, from Figure 4, some data about our crucial homotopy classes. 

\begin{lem}\label{theuhrclassesG48} There is an isomorphism
$$
\ZZ/2 \cong \pi_{45} E^{hG_{48}}.
$$
The generator is detected by the class
$$
\Delta\kappabar\eta \in H^5(G_{48},E_{50}).
$$
The class $\Delta\kappabar^2\eta^2\in H^{10}(G_{48},E_{76})$ is a non-zero permanent
cycle detecting a  generator of the subgroup $\pi_{66}E^{hG_{48}}$ of the pure $K(2)$-classes
of that degree. 
\end{lem}

We close with some remarks on the role of $\mu$ in the $d_3$ differentials.
\begin{lem}\label{muhelps} Let $\mu \in H^1(\GG_2,E_6)$ be the image of the class
$$
\alpha_3 \in \Ext_{BP_\ast BP}^1(\Sigma^6 BP_\ast, BP_\ast).
$$
Then in any of
the Adams-Novikov Spectral Sequences 
$$
H^\ast(F,E_\ast) \Longrightarrow \pi_{t-s}E^{hF}
$$
and for any $x \in H^\ast(F,E_\ast)$ we have 
$$
d_3(x\mu) = d_3(x)\mu + x\eta^4.
$$
In the spectral sequences
\begin{equation}\label{mod2ANSSstandard}
H^\ast(F,E_\ast/2) \Longrightarrow \pi_{t-s}(E^{hF}\wedge S/2)
\end{equation}
we have
$$
d_3(v_1^2x) = d_3(x)v_1^2 + x\eta^3 + y
$$
where $y\eta=0.$ Finally, in the spectral sequence (\ref{mod2ANSSstandard}) we have
$$
d_3(v_1^4x) = v_1^4d_3(x)\ .
$$
\end{lem}

\begin{proof} In the Adams-Novikov spectral sequence
$$
\Ext^s_{BP_\ast BP}(\Sigma^t BP_\ast,BP_\ast) \Longrightarrow \ZZ_{(2)}\otimes \pi_{t-s}S^0.
$$ 
we have $d_3(\alpha_3) = \eta^4$. (In fact, by \cite{MRW}, Corollary 4.23, 
$\eta^4 \ne 0$ at $E_2$ and $E_2^{1,6} \cong \ZZ/2$ generated
by $\alpha_3$. The differential is then forced.) 
Since the fixed point spectral sequence is a
module over this standard Adams-Novikov Spectral Sequence, the first formula follows.
The second formula follows because $v_1^2\eta = \alpha_3$ in
$$
\Ext^s_{BP_\ast BP}(\Sigma^t BP_\ast,BP_\ast/2).
$$
The third formula follows from the fact that $S/2$ has a $v_1^4$-self map. 
\end{proof} 

\section{Algebraic and topological resolutions}\label{ch:Tower}

In this section we review the centralizer resolution constructed by Hans-Werner Henn
\cite{HennRes} \S 3.4 and then begin the construction of the topological duality resolution. The details of the algebraic duality resolution can be found in \cite{beaudry}.

The two resolutions have complementary features. While we will not try to make
this thought completely precise, the duality resolution reflects, in an essential
way, the fact that the group $\SS_2^1$ is a virtual Poincar\'e duality group of dimension $3$. The 
centralizer resolution on the other hand, is much closer to being an Adams-Novikov tower as there
is an underlying relative homological algebra in the spirit of Miller \cite{Miller}. See
Remark \ref{F-res} below.

\subsection{The centralizer resolution}

Henn's centralizer resolutions grew out of his paper \cite{HennDuke} which used the 
centralizers of elementary abelian subgroups of $S_n$ to detect elements in the 
cohomology of $S_n$. At the prime $2$, this approach needs a slight modification,
as the maximal finite $2$-group in $S_2$ is $Q_8$, which is not elementary abelian. 

\begin{rem}\label{2-g24s} In (\ref{ss-curve1}) we defined $G_{24} \subseteq \SS_2^1$
as the image of a group of automorphisms of a supersingular elliptic curve. 
The group
$\SS_2^1$ fits into a short exact sequence
$$
\xymatrix{
1 \longr \SS_2^1 \rto &  \SS_2 \rto^-N & \ZZ_2 \longr 1
}
$$
where $N$ is the reduced determinant map of (\ref{norm-defined}). 
Let $$\pi=1+2\omega$$ be an element of $\mathbb{S}_2,$ where $\omega \in \mathbb{W}^{\times}$ is a cube root 
of unity. Notice that $\pi$ is not an element of $\mathbb{S}_2^1$ because $N(\pi)=3.$
Then  we define $G'_{24} := \pi G_{24}\pi^{-1} \subseteq \SS_2^1.$ This is a subgroup
isomorphic to $G_{24},$ but not conjugate to $G_{24}$ in $\SS_2^1$. 

Note that multiplication by $\pi$ defines an equivalence $E^{hG_{24}} \simeq
E^{hG'_{24}}$. For complete details on this and more,
see \cite{beaudry}. 
\end{rem}

We now have the following result from \S 3.4 of \cite{HennRes}. This is the {\it algebraic
centralizer resolution}.

\begin{thm}\label{cent-res-thm} There is an exact sequence of continuous $\SS_2^1$-modules
\begin{align}\label{cent-res-alg}
0 \to \ZZ_2[[\SS_2^1/C_6]] \to  \ZZ_2[[\SS_2^1/C_2]]
\to &\ZZ_2[[\SS_2^1/C_6]] \oplus \ZZ_2[[\SS_2^1/C_4]]\nonumber\\
&\to \ZZ_2[[\SS_2^1/{G_{24}}]] \oplus \ZZ_2[[\SS_2^1/G'_{24}]]  \mathop{\longr}^{\epsilon} \ZZ_2 \to 0\ .
\end{align}
The map $\epsilon$ is the sum of the augmentation maps.
\end{thm}

We will call this a resolution, even though  the terms are not projective as $\ZZ_2[[\SS_2^1]]$-
modules.  It is an $\cF$-projective resolution, an idea we explore below in Remark \ref{F-res}.

\begin{rem}\label{cent-alg-ss}  Suppose we write
\begin{align*}\label{cent-fibers}
P_0 &= \ZZ_2[[\SS_2^1/G_{24}]] \times \ZZ_2[[\SS_2^1/G_{24}']]\\
P_1 &= \ZZ_2[[\SS_2^1/C_6]] \times \ZZ_2[[\SS_2^1/C_4]]\\
P_2 &= \ZZ_2[[\SS_2^1/C_2]]\nonumber\\
P_3 &= \ZZ_2[[\SS_2^1/C_6]]\nonumber.
\end{align*}
Then for any profinite $\SS_2^1$-module $M$, we get
a spectral sequence
$$
E_1^{p,q} \cong \Ext^q_{\ZZ_2[[\SS_2^1]]}(P_p,M) \Longrightarrow H^{p+q}(\SS_2^1,M).
$$
The $E_1$-terms can all be written as group cohomology groups; for example
$$
E_1^{0,q} \cong H^q(G_{24},M) \times H^q(G_{24}',M).
$$
We will call this the {\it algebraic centralizer resolution spectral sequence}. In many
applications, the distinction between the groups $G_{24}$ and $G_{24}'$ disappears. 
For example, if $M$ is
a $\GG_2$-module (such as $E_nX$ for some spectrum $X$) then multiplication
by $\pi$ induces an isomorphism $H^q(G_{24},M) \cong H^q(G_{24}',M)$.
\end{rem} 

\begin{rem}\label{morava-cent} We can induce the resolution (\ref{cent-res-alg}) of $\SS_2^1$-modules
up to a resolution of $\GG_2$-modules and obtain an exact sequence
\begin{align*}
0 \to \ZZ_2[[\GG_2/C_6]] \to  \ZZ_2&[[\GG_2/C_2]]
\to \ZZ_2[[\GG_2/C_6]] \oplus \ZZ_2[[\GG_2/C_4]]\nonumber\\
&\to \ZZ_2[[\GG_2/G_{24}]] \oplus \ZZ_2[[\GG_2/G_{24}]]  \to \ZZ_2[[\GG_2/\SS_2^1]] \to 0\ .
\end{align*}
Since $G'_{24}$ is conjugate to $G_{24}$ in $\GG_2$, we have $\ZZ_2[[\GG_2/G_{24}]] \cong
\ZZ_2[[\GG_2/G'_{24}]]$ as $\GG_2$-modules and we have made that substitution.
If $F$ is any closed subgroup of $\GG_2$, then the equivalence of (\ref{theusual}) gives us
an isomorphism of twisted $\GG_2$-modules
$$
\Hom_{\ZZ_2[[\GG_2]]}(\ZZ_2[[\GG_2/F]],E_\ast) \cong E_\ast E^{hF}.
$$
Combining these observations, we a get an exact sequence of twisted $\GG_2$-modules
\begin{equation}\label{cent-res-morava}
0\to E_\ast E^{h\SS_2^1}\rightarrow \begin{array}{c} E_\ast E^{hG_{24}}\\
\vspace{.05in} 
\times \\ \vspace{.05in} E_\ast E^{hG_{24}} \end{array} \rightarrow
\begin{array}{c} E_\ast E^{hC_6} \\
\vspace{.05in} 
\times \\ \vspace{.05in}E_\ast E^{hC_4} \end{array} \rightarrow
\begin{array}{c}E_\ast E^{hC_2}
\end{array} \rightarrow 
\begin{array}{c}
E_\ast E^{hC_6} \end{array} \to 0\ .
\end{equation}
\end{rem}

We then have the following result; this is the {\it topological centralizer resolution} of Theorem 12 of \cite{HennRes}.

\begin{thm}\label{cent-res} The algebraic resolution of \ref{cent-res-morava} can
be realized by a sequence of spectra
\begin{equation*}
E^{h\SS_2^1} \mathop{\longrightarrow}^{p} \begin{array}{c} E^{hG_{24}}\\
\vspace{.05in} 
\times \\ \vspace{.05in} E^{hG_{24}} \end{array} \rightarrow
\begin{array}{c} E^{hC_6} \\
\vspace{.05in} 
\times \\ \vspace{.05in} E^{hC_4} \end{array} \rightarrow
\begin{array}{c} E^{hC_2}
\end{array} \rightarrow 
\begin{array}{c}
E^{hC_6} \end{array}
\end{equation*}
All compositions and all Toda brackets are zero modulo indeterminacy.
\end{thm}

\begin{rem}\label{twin tower} The vanishing of the Toda brackets in this result has
several implications. To explain these and for future reference we write, echoing the
notation of Remark \ref{cent-alg-ss}:
\begin{align}\label{cent-fibers}
F_0 &= E^{hG_{24}} \times E^{hG_{24}}\nonumber\\
F_1 &= E^{hC_6} \times E^{hC_4}\\
F_2 &= E^{hC_2}\nonumber\\
F_3 &= E^{hC_6}\nonumber\ .
\end{align}
Then the resolution of Theorem  \ref{cent-res} can be refined to a tower of fibrations under
$E^{h\SS_2^1}$:
\begin{equation}\label{cent-tower}
\xymatrix{
E^{h\SS_2^1} \rto & Y_2 \rto & Y_1 \rto& E^{hG_{24}}\times E^{hG_{24}}=F_0 \\
\Sigma^{-3}F_3 \uto
& \Sigma^{-2}F_2 \uto &
\Sigma^{-1}F_1 \uto
}
\end{equation}
Alternatively we could refine the resolution into a tower over $E^{h\SS_2^1}$. Let us write
$$
\xymatrix@C=10pt@R=12.5pt{
X \ar[dr] && \ar@{-->}[ll]Z\\
&Y\ar[ur]
}
$$
for a cofiber sequence (i.e., a triangle) $X \to Y \to Z \to \Sigma X$. Then we have a diagram of cofiber sequences
\begin{equation}\label{cent-cofs}
\xymatrix@C=10pt@R=12.5pt{
E^{h\SS_2^1} \ar[dr]_p && \ar@{-->}[ll] C_1\ar[dr] && \ar@{-->}[ll] C_2\ar[dr] && \ar@{-->}[ll] C_3\ar[dr]^\simeq\\
&F_0\ar[ur]&&F_1\ar[ur]&&F_2\ar[ur]&&F_3
}
\end{equation}
where each of the compositions $F_{i-1} \to C_i \to F_i$
is the map $F_{i-1} \to F_i$ in the resolution. 

The towers (\ref{cent-tower}) and (\ref{cent-cofs}) determine each other. This is because
there is a diagram with rows and columns cofibration sequences
\begin{equation}\label{the-same}
\xymatrix{
\Sigma^{-(s+1)} C_{s+1} \rto \dto & E^{h\SS_2^1} \rto \dto & Y_s\dto \\
\Sigma^{-s} C_{s} \rto \dto & E^{h\SS_2^1} \rto \dto & Y_{s-1}\dto \\
\Sigma^{-s}F_s \rto & \ast \rto &\Sigma^{-s+1}F_s
}
\end{equation}
Using the tower over $E^{h\SS_2^1}$ of (\ref{cent-cofs}) we get a
number of spectral sequences; for example, if $Y$ is any spectrum, we get a
spectral sequence for the function spectrum $F(Y,E^{h\SS_2^1})$
\begin{equation}\label{cent-tower-ss}
E_1^{s,t} = \pi_tF(Y,F_s) \Longrightarrow \pi_{t-s}F(Y,E^{h\SS_2^1}).
\end{equation}
Up to isomorphism, this spectral sequence can be obtained from the tower of (\ref{cent-tower});
this follows from (\ref{the-same}). 
\end{rem}

\begin{rem}\label{homology-cent-res} It is direct to calculate $E_\ast C_s$ and $E_\ast Y_s$
for the layers of the two towers. If we define $K_s \subseteq E_\ast F_s$ to be the image
of $E_\ast F_{s-1} \to E_\ast F_s$, then $E_\ast C_s \cong K_s$ and, more, if we apply
$E_\ast$ to (\ref{cent-cofs}) we get a collection of short exact sequences:
$$
\xymatrix@C=10pt@R=15pt{
E_\ast E^{h\SS_2^1} \ar@{>->}[dr] && \ar@{-->}[ll]_0 E_\ast C_1\ar@{>->}[dr] && \ar@{-->}[ll]_0 E_\ast C_2\ar@{>->}[dr] && \ar@{-->}[ll]_0 E_\ast C_3\ar[dr]^\cong\\
&E_\ast F_0\ar@{->>}[ur]&&E_\ast F_1\ar@{->>}[ur]&&E_\ast F_2\ar@{->>}[ur]&&E_\ast F_3\ .
}
$$
Notice that this implies that each of the dotted arrows of (\ref{cent-cofs}) has Adams-Novikov filtration
one.  Finally, the cofibration sequence $E^{h\SS_2^1} \to Y_s \to \Sigma^{-s}C_{s+1}$ induces a
short exact sequence
$$
0 \to E_\ast E^{h\SS_2^1} \to E_\ast Y_s \to \Sigma^{-s}K_{s+1} \to 0\ .
$$
\end{rem}

\subsection{The duality resolution, first steps}

We have the {\it algebraic duality resolution } from  \cite{beaudry}. The groups $G_{24}$ and $G_{24}'$
are defined in Remark \ref{2-g24s}.

\begin{thm}\label{dual-res-alg}There is an exact sequence of continuous $\SS_2^1$-modules
\begin{align}\label{duality-res}
0 \to \ZZ_2[[\SS_2^1/G_{24}']] \to  \ZZ_2[[\SS_2^1/C_6]]
\to \ZZ_2&[[\SS_2^1/C_6]] \to
\ZZ_2[[\SS_2^1/G_{24}]]  \mathop{\to}^{\epsilon} \ZZ_2 \to 0
\end{align}
where $\epsilon$ is the augmentation. 
\end{thm}

\begin{rem}\label{dual-alg-ss}  As in Remark \ref{cent-alg-ss} we get a spectral sequence. Suppose we write
\begin{align*}\label{cent-fibers}
Q_0 &= \ZZ_2[[\SS_2^1/G_{24}]]\\
Q_1 &=Q_2 = \ZZ_2[[\SS_2^1/C_6]]\\
Q_3 &= \ZZ_2[[\SS_2^1/G_{24}']]
\end{align*}
Then for any profinite $\SS_2^1$-module $M,$ such as $E_*X=(E_2)_*X$ for some finite spectrum $X$, we get
a spectral sequence
$$
E_1^{p,q} \cong \Ext^q_{\ZZ_2[[\SS_2^1]]}(Q_p,M) \Longrightarrow H^{p+q}(\SS_2^1,M),
$$
which we will call the {\it algebraic duality resolution spectral sequence}.
\end{rem} 

As in Remark \ref{morava-cent} and (\ref{cent-res-morava}) we immediately have the
following consequence.

\begin{cor}\label{morava-dual} There is an exact sequence of twisted $\GG_2$-modules
\begin{equation*}
0 \to E_\ast E^{h\SS_2^1}\to E_\ast E^{hG_{24}} \to 
E_\ast E^{hC_6} \to E_\ast E^{hC_6} \to E_\ast E^{hG_{24}}\to 0.
\end{equation*}
The first of these maps is induced by the map on homotopy fixed point spectra $E^{h\SS_2^1} \to E^{hG_{24}}$
induced by the subgroup inclusion $G_{24} \subseteq \SS_2^1$.
\end{cor}

We'd now like to prove the following result, paralleling Theorem \ref{cent-res}; it also appears
in \cite{HennRes}. The main
work of the next two sections and, indeed, the main theorem of this paper is to identify $X$.

\begin{prop}\label{dual-res} The algebraic resolution of \ref{morava-dual} can
be realized by a sequence of spectra
\begin{equation*}
E^{h\SS_2^1} \mathop{\longrightarrow}^q E^{hG_{24}} \rightarrow E^{hC_6} \rightarrow
E^{hC_6} \rightarrow X
\end{equation*}
with $E_\ast X \cong E_\ast E^{hG_{24}}$ as a twisted $\GG_2$-module.
All compositions and all Toda brackets are zero modulo indeterminacy.
\end{prop}

\begin{rem}\label{start-of-the-proof} A consequence of the last sentence of this result is that
this resolution can be refined to a tower of fibrations under $E^{h\SS_2^1}$
\begin{equation}\label{duality-tower}
\xymatrix{
E^{h\SS_2^1} \rto & Z_2 \rto & Z_1 \rto& E^{hG_{24}} \\
\Sigma^{-3}X \uto
& \Sigma^{-2}E^{hC_6} \uto &
\Sigma^{-1}E^{hC_6} \uto
}
\end{equation}
or to a tower over $E^{h\SS_2^1}$
\begin{equation}\label{duality-cofs}
\xymatrix@C=10pt@R=12.5pt{
E^{h\SS_2^1} \ar[dr]_q && \ar@{-->}[ll] D_1\ar[dr] && \ar@{-->}[ll] D_2\ar[dr] && \ar@{-->}[ll] D_3\ar[dr]^\simeq\\
&E^{hG_{24}}\ar[ur]&&E^{hC_6}\ar[ur]&&E^{hC_6}\ar[ur]&&X\ .
}
\end{equation}
As in Remark \ref{homology-cent-res} the dotted arrows have Adams-Novikov filtration 1.
Examining this last diagram, we see that $X$ can be {\it defined} as the cofiber of 
$D_2 \to E^{hC_6}$ and it will follow, as in Remark \ref{homology-cent-res}, that
$E_\ast X \cong E_\ast E^{hG_{24}}$. Thus Proposition \ref{dual-res} is equivalent to the following result. See also \cite{HennRes} or \cite{thesis}.
\end{rem}

\begin{lem}\label{dual-res-lem} The truncated resolution
\begin{equation*}
0 \to E_\ast E^{h\SS_2^1}\to E_\ast E^{hG_{24}} \to 
E_\ast E^{hC_6} \to E_\ast E^{hC_6}
\end{equation*}
can be realized by a sequence of spectra
$$
E^{h\SS_2^1} \longr E^{hG_{24}} \to E^{hC_6} \to E^{hC_6}
$$ 
such that all compositions are zero and the one Toda bracket is zero modulo indeterminacy. 
\end{lem}

\begin{proof}
The map $E^{h\SS_2^1} \xrightarrow{q} E^{hG_{24}}$ is induced by the inclusion map on homotopy fixed point spectra 
induced by the subgroup inclusion $G_{24} \subseteq \SS_2^1$. To realize the 
other maps and to show the compositions are zero, we prove that Hurewicz map
$$
\pi_0F(E^{hF},E^{hC_6}) \longr \Hom_{Mor}(E_0E^{hF},E_0E^{hC_6})
$$
to the category of Morava modules is an isomorphism for $F=\SS_2^1$ or $F=G_{24}$. To see this, first
note there is an isomorphism
$$
\pi_0F(E^{hF},E^{hC_6}) = \pi_0E[[\GG_2/F]]^{hC_6} \cong H^0(C_6,E_0[[\GG_2/F]]).
$$
This follows from (\ref{thegencase-homotopy}) and the fact that $\pi_0E^{hK} = H^0(K,E_0$) whenever
$K \subseteq C_6$; see \S \ref{homotopyC2}. Then we can finish the argument by using (\ref{EstarEF}) and 
(\ref{mappingf}) to show
\begin{align*}
H^0(C_6,E^0[[\GG_2/F]]) &\cong
\Hom_{E_0[[\GG_2]]}(E_0[[\GG_2/C_6]],E_0[[\GG_2/F]])\\
&\cong  \Hom_{Mor}(E_0E^{hF},E_0E^{hC_6})\ .
\end{align*}

This leaves the Toda bracket. To see that it is zero modulo indeterminacy we show that the indeterminacy is the
entire group. To be specific, we show that the inclusion $E^{h\SS_2^1} \to E^{hG_{24}}$
induces a surjection
$$
\pi_\ast F(E^{hG_{24}},E^{hC_6}) \longr \pi_\ast F(E^{h\SS_2^1},E^{hC_6}).
$$
Using (\ref{mappingf}), we can rewrite this map as 
$$
\pi_\ast E[[\GG_2/G_{24}]]^{hC_6} \longr \pi_\ast E[[\GG_2/\SS_2^1]]^{hC_6}.
$$
Again using (\ref{mappingf}), the inclusion $E^{h\SS_2^1} \to E^{hG_{24}}$ induces a map of
$C_6$-spectra
$$
E[[\GG_2/G_{24}]] \simeq F(E^{hG_{24}},E) \to F(E^{h\SS_2^1},E) \simeq E[[\GG_2/\SS_2^1]].
$$
It is thus sufficient to show that the quotient map on cosets
$$
\GG_2/G_{24} \longr \GG_2/\SS_2^1
$$
has a $C_6$-splitting. Since $\GG_2 \cong \SS_2 \rtimes \Gal(\FF_4/\FF_2)$, every coset in $\GG_2/\SS_2^1$
has a representative of form $\pi^i\phi^\epsilon\SS_2^1$ where $\pi$ is as in Remark \ref{2-g24s}, $\phi
\in \Gal(\FF_4/\FF_2)$ is the Frobenius, $i \in \ZZ_2$, and $\epsilon =0$ or $1$. The splitting is then given by
$$
\pi^i\phi^\epsilon\SS_2^1 \longmapsto \pi^i\phi^\epsilon G_{24}.
$$
\end{proof} 

\subsection{Comparing the two resolutions} There is a map from the
centralizer tower to the duality tower; we will not prove that here.
In the end we will only need a small part of the data given by such a map, and what we need
is in Remark \ref{compare-ss}. 

\begin{rem}\label{F-res} The underlying algebra for the centralizer resolutions fits well
the relative homological algebra usually deployed in building an Adams-Novikov tower;
this goes back to Miller in \cite{Miller}, among other sources.

Here is more detail. Let $\cF$ be the set of conjugacy classes of finite subgroups
of $\SS_2^1$. A continuous $\SS_2^1$-module $P$ is {\it $\cF$-projective} if the natural map
$$
\bigoplus_{F \in \cF} \ZZ_2[[\SS_2^1]] \otimes_{\ZZ_2[[F]]} P \longr P
$$ 
is split surjective, where $F$ runs over representatives for the classes in $\cF$. The class of $\cF$-projectives
is the smallest class of continuous $\SS_2^1$-modules closed under finite sums, retracts, and containing all induced 
modules $\ZZ_2[[\SS_2^1]] \otimes_{\ZZ_2[[F]]} M$, where $M$ is a continuous $F$-module. 

The class of $\cF$-projectives defines a class of $\cF$-exact morphisms, there are enough $\cF$-projectives,
there are $\cF$-projective resolutions, and so on. All of this and more is discussed in \S 3.5 of \cite{HennRes}.
\end{rem}

Comparing the two towers now begins with the following result, see Remark (d) after Proposition 17 in \cite{HennRes}.

\begin{prop}\label{cent-is-F} The centralizer resolution (\ref{cent-res-alg}) is an $\cF$-projective resolution
of the trivial $\SS_2^1$-module $\ZZ_2$. 
\end{prop} 

Thus if we write $P_\bullet \to \ZZ_2$ for the centralizer resolution (\ref{cent-res-alg}) and $Q_\bullet \to \ZZ_2$
for the duality resolution (\ref{duality-res}), then standard homological algebra gives
us a map of resolutions, unique up to chain homotopy
$$
\xymatrix{
Q_\bullet \rto \dto_{g_\bullet} & \ZZ_2 \dto^=\\
P_\bullet \rto &\ZZ_2\ .
}
$$
The map $g_0: Q_0 \to P_0$
can be chosen to be the inclusion onto the first factor

\begin{equation}\label{fzerois}
\xymatrix{
Q_0 = \ZZ_2[[\SS_2^1/G_{24}]] \rto^-{i_1} &  \ZZ_2[[\SS_2^1/G_{24}]] \oplus  \ZZ_2[[\SS_2^1/G'_{24}]]=P_0\ .
}
\end{equation}

\begin{rem}\label{compare-ss} 
This immediately gives a map from the centralizer resolution spectral sequence of Remark \ref{cent-alg-ss} to the 
duality resolution spectral sequence of Remark \ref{dual-alg-ss}. This map is independent of the choice of $g_\bullet$ 
at the $E_2$-page. This can be lifted to a map from the centralizer tower to the duality tower,
although we don't need that here and won't prove it. We note that (\ref{fzerois})
implies there is a commutative diagram where the horizontal maps are the
edge homomorphisms of the two spectral sequences
$$
\xymatrix{
H^\ast(\SS_2^1,E_\ast) \rto^-{p^\ast} \dto_= & H^\ast(G_{24},E_\ast) \times H^\ast(G_{24},E_\ast) \dto^{(g_0)^\ast}\\
H^\ast(\SS_2^1,E_\ast) \rto_-{q^\ast}  & H^\ast(G_{24},E_\ast)
}
$$
and $(g_0)^\ast$ is projection onto the first factor. This can be realized by a diagram of spectra, where the
map $g_0$ is again projection onto the first factor
$$
\xymatrix{
E^{h\SS_2^1}\rto^-{p} \dto_= & E^{hG_{24}}  \times E^{hG_{24}}\dto^{g_0}\\
E^{h\SS_2^1} \rto_-{q}  & E^{hG_{24}}\ .
}
$$
\end{rem}

\section{Constructing elements in $\pi_{192k+48}X$}\label{ch:homotopy}

We now turn to the analysis of the homotopy groups of $X$, where $\Sigma^{-3}X$ is the 
top fiber in the duality tower; see Proposition \ref{dual-res}. We have an isomorphism of Morava modules
$E_\ast X \cong E_\ast E^{hG_{24}}$ and hence, by Proposition \ref{e2-is-grp-cohio} and
Lemma \ref{reduce-to-finite}, a spectral sequence
$$
H^\ast (G_{24},E_\ast) \Longrightarrow \pi_\ast X.
$$
See also Remark \ref{rem:why-did-we-bother}. 
The cohomology if $G_{24}$ is discussed in Theorem \ref{coh-G48}. 
In this section we show, roughly, that $\Delta^{8k+2} \in H^0(G_{24},E_{192k + 48})$ is a permanent
cycle---which would certainly be necessary if our main result is true.
The exact result is below in Corollary \ref{describer}. In the next section, we will use this and a mapping
space argument to finish the identification of the homotopy type $X$. 

The statements and the arguments in this section have a rather fussy nature because the
spectrum $X$ has no {\it a priori} ring or module structure and, in particular, the $\WW$-algebra structure on
$H^\ast(G_{24},E_\ast)$ does not immediately extend to a $\WW$-module structure on $\pi_\ast X$.

The results of this section were among the main results in the first author's thesis \cite{thesis} and the key
ideas for the entire project can be found there.  

We begin by combining Remark \ref{coh-implies} and Lemma \ref{theuhrclassesG48} to obtain the following result. 
Note that $45 \equiv 5$ modulo $8$, so there is no contributions
from the $bo$-patterns in that degree. In all degrees the $bo$-patterns lie in Adams-Novikov filtration at most
$2$. The following is an immediate consequence of Lemma \ref{more-split} and Lemma \ref{theuhrclassesG48}.

\begin{lem}\label{theuhrclasses} There is an isomorphism
$$
\FF_4 \cong \pi_{45} E^{hG_{24}}.
$$
We can chose an $\FF_4$ generator detected by the class
$$
\Delta\kappabar\eta \in H^5(G_{24},E_{50}).
$$
The class $\Delta\kappabar^2\eta^2 \in H^{10}(G_{24},E_{76})$ is a non-zero permanent
cycle detecting an $\FF_4$ generator of the subgroup of $\pi_{66}E^{hG_{24}}$ consisting of the elements
of Adams-Novikov filtration greater than $2$.
\end{lem}

Let $p: E^{h\SS_2^1} \to E^{hG_{24}} \times E^{hG_{24}}$ be the augmentation in the
topological centralizer resolution of Theorem \ref{cent-res}. This is the same map as from
the top to  the bottom of the centralizer resolution tower (\ref{cent-tower}).

\begin{lem}\label{thefirstcalculation} Let $k \in \ZZ$. The map
$$
p_\ast:\pi_{192k+45}E^{h\SS_2^1} \longr \pi_{192k+45} (E^{hG_{24}} \times E^{hG_{24}})
$$
is surjective. If $x \in \pi_{192k+45}E^{h\SS_2^1}$ has the property that $p_\ast(x) \ne 0$,
then $x$ has Adams-Novikov filtration at most $5$, $x\kappabar\eta \ne 0$, and
$x\kappabar\eta$ is detected by a class of  Adams-Novikov filtration  at most 10.
\end{lem}

\begin{proof} For the first statement we examine the homotopy spectral sequence of the centralizer
tower (\ref{cent-tower-ss}). In this case this spectral sequence reads
$$
\pi_tF_s \Longrightarrow \pi_{t-s}E^{h\SS_2^1}.
$$
The fibers $F_s$ are spelled out in (\ref{cent-fibers}). We are asking that the edge homomorphism
$$
p_\ast:\pi_{192k+45}E^{h\SS_2^1} \longr \pi_{192k+45}F_0
$$
be surjective. The crucial input is that
$$
\pi_kE^{hC_6} \subseteq \pi_kE^{hC_2} = 0
$$
for $k=45$, $46$, and $47$ and that $\pi_{45}E^{hC_4}=0$. See Proposition \ref{homotopyc6}
and Figure 2.

The final statement follows from Lemma \ref{theuhrclasses}.
\end{proof}

Now let $q: E^{h\SS_2^1} \to E^{hG_{24}}$ be the augmentation in the
topological duality resolution of Proposition \ref{dual-res}. This is also the projection
from the top to the bottom of the duality resolution tower (\ref{duality-tower}).
Let $i:\Sigma^{-3}X \to E^{h\mathbb{S}_2^1}$ be the map from the top fiber of the duality tower.
Consider the commutative diagram
\begin{equation}\label{diagramr}
\xymatrix{
\pi_{192k+ 45}\Sigma^{-3}X \rto^{i_\ast}\dto_{r} & \pi_{192k+ 45} E^{h\SS_2^1} \rto^-{=}\dto^{p_\ast} &
\pi_{192k+ 45} E^{h\SS_2^1} \dto^{q_\ast} \\
\FF_4 \rto & \pi_{192k+ 45}(E^{hG_{24}} \times E^{hG_{24}}) \rto_-{(g_0)_\ast} & \pi_{192k+ 45}E^{hG_{24}}.
}
\end{equation}
The bottom row is short exact and induced by the map between the resolutions. See Remark \ref{compare-ss}.
The map $r$  is defined by this diagram and the fact that the composition 
$$
\pi_\ast \Sigma^{-3}X \longr \pi_\ast E^{h\SS_2^1} \mathop{\longr}^{q_\ast} \pi_\ast E^{hG_{24}}
$$
is zero.

\begin{prop}\label{thecrucialelement} The map
$$
r:\pi_{192k+45}\Sigma^{-3}X \to \FF_4
$$
is surjective. If $y \in \pi_{192k+45}\Sigma^{-3}X$ is any class so that
$r(y) \ne 0$, then $y$ is detected by a class
$$
f\defeq f(j)\Delta^{8k+2} \in H^0(G_{24},E_{192k+48}) \cong W[[j]]\Delta^{8k+2}.
$$
Furthermore
$$
f\kappabar\eta \in H^5(G_{24},E_{192k+73})
$$
is a non-zero permanent cycle in the spectral sequence for $\pi_\ast X$.
\end{prop}

\begin{proof} In the diagram (\ref{diagramr}), the map $p_\ast$ is onto, and any element
$x$ in the kernel of $(g_0)_\ast$ must have filtration at least $1$ in the homotopy spectral
sequence of the duality tower. Since $\pi_kE^{hC_6}=0$ for $k=46$ and $k=47$, by Proposition
\ref{homotopyc4},  any such element must be the image of class $y\in \pi_\ast \Sigma^{-3}X$.
This shows $r$ is  surjective.

The map $\Sigma^{-3}X \to E^{h\SS_2^1}$ raises Adams-Novikov filtration by $3$; see the diagram of
\eqref{duality-cofs} and the remarks thereafter. 
If $r(y) = p_\ast i_\ast(y) \ne 0$, then by Lemma \ref{thefirstcalculation}, $y$ must have
Adams-Novikov filtration at most 2; however, by Theorem \ref{coh-G48} and the chart of Figure 3
we have that 
$$
H^1(G_{24},E_{192k+49})=0=H^2(G_{24},E_{192k+50}).
$$
Thus $y$ must have filtration $0$.
Similarly $0 \ne y\kappabar\eta$ must have filtration at least five and at most 7. Again we
examine the chart of Figure 3 to find it must have filtration 5 and be detected in
group cohomology, as claimed.
\end{proof}

Recall that we are writing $S/2$ for the mod $2$ Moore spectrum. 

\begin{prop}\label{mod2} Let $y \in \pi_{192k+48}X$ be detected by
$$
f=f(j)\Delta^{8k+2}\in H^0(G_{24},E_\ast).
$$
If $r(y)\ne 0$, then $f$ and $f\kappabar\eta$
are non-zero permanent cycles in the spectral sequence
$$
H^\ast(G_{24},E_\ast/2) \Longrightarrow \pi_\ast (X \wedge S/2).
$$
\end{prop} 

\begin{proof} This follows from Proposition \ref{thecrucialelement} and the fact that
$$
H^5 (G_{24},E_{192k+ 73}) \to H^5 (G_{24},E_{192k+73}/2) 
$$
is injective. This last statement can be deduced from the long exact sequence in cohomology induced
by the short exact sequence $0 \to E_\ast \to E_\ast \to E_\ast/2 \to 0$ and the fact that
$$
H^6(G_{24},E_{192k+ 73}) = 0\ .
$$
See Figure 3.
\end{proof}

    The crucial theorem then becomes:

\begin{thm}\label{analysisak} Let $y \in \pi_{192k+48}X$ and let
$$
f=f(j)\Delta^{8k+2}\in H^0(G_{24},E_\ast)
$$
be the image of $y$ under the edge homomorphism
$$
\pi_\ast X \longr H^0(G_{24},E_\ast).
$$
If $f(j) \equiv 0$ modulo $(2,j)$, then $r(y) = 0$.
\end{thm}

\begin{proof} We will show that if $f(0) \equiv 0$ modulo $2$, then $f\kappabar = 0$ in
$E_4^{\ast,\ast}(X \wedge S/2)$. We can then apply Proposition \ref{mod2}.

Under the assumption $f(0) \equiv 0$ modulo $2$ we have 
$$
f = jg(j)\Delta^{8k+2} \in H^0(G_{24},E_\ast/2).
$$
We will show that in the Adams-Novikov Spectral Sequence for $L_{K(2)}(X\wedge S/2)$ we have
$$
d_3(v_1^8g(j)\Delta^{8k+2}\mu) = f\kappabar\ .
$$
The result will follow.

We appeal to Theorem \ref{coh-G48}, Remark \ref{whatsupwithj},
and the chart of Figure 3. We have 
$$
j\kappabar = c_4^2\eta^4 = v_1^8\eta^4
$$
and, hence, that
$$
f\kappabar  = jg(j)\Delta^{8k+2}\kappabar  = v_1^8g(j)\Delta^{8k+2}\eta^4.
$$
Since $f\kappabar$ is a $d_3$-cycle, $d_3$ is $\eta$-linear, and
$$
\eta^4:E_3^{192k+48,0}/2 \cong H^{0}(G_{24},(E/2)_{192k+48})
\to H^{4}(G_{24},E_{192k+52}) \cong E_3^{192k+48+4,8}
$$ 
is injective we have that
$$
d_3(v_1^8g(j)\Delta^{8k+2})= 0.
$$
It now follows from Lemma \ref{muhelps} that
$$
d_3(v_1^8g(j)\Delta^{8k+2}\mu) = v_1^8g(j)\Delta^{8k+2}\eta^4= f\kappabar.
$$
This is what we promised.
\end{proof}

The next result has a slightly complicated statement because we don't know yet that
$\pi_\ast X$ is a $\WW$-module.

\begin{cor}\label{describer} There is a commutative diagram
$$
\xymatrix{
\pi_{192k+48}X \rto \dto_r & H^0(G_{24},E_{192k+48}) \dto^{\epsilon}\\
\FF_4 \rto_{\cong} & \FF_4
}
$$
where the bottom map is some possibly non-trivial isomorphism of groups and
$$
\epsilon(f(j)\Delta^{8k+2}) = f(0)\qquad\mathrm{mod}\ (2,j).
$$
There are homotopy classes $x_{k,i} \in \pi_{192k+48}X$, $i=1,2$ detected by classes
$$
f_i(j)\Delta^{8k+2} \in H^0(G_{24},E_{192k+48})
$$
so that $f_1(0)$ and $f_2(0)$
span $\FF_4$ as an $\FF_2$ vector space. 
\end{cor}

\begin{proof} This is an immediate consequence of Theorem \ref{analysisak}.
\end{proof}

\begin{rem}\label{whatisthisclass} Lemma \ref{thecrucialelement} produces  classes
$f \in \pi_{45}\Sigma^{-3}X$ for which $f\kappabar\eta \ne 0$.
The image of any such $f$ in $\pi_{45}E^{h\SS_2^1}$ is non-zero and has Adams-Novikov filtration
at least $3$. It is natural to ask what we know about these classes. In particular, does one
of these classes come from the homotopy groups of sphere itself under the unit map
$\pi_\ast S^0 \to \pi_\ast E^{h\SS_2^1}$? 

There are classes $x \in \pi_{45}S^0$ detected in the Adams Spectral Sequence by
$h_4^3 = h_3^2h_5$. Note any such class has filtration $3$.
There is a choice for $x$ that seem to support non-zero multiplications by many of the basic
elements in $\pi_\ast S^0$, including $\eta$ and $\kappabar$. It would be easy to guess
that this class is mapped to the class we have constructed, but we've not yet settled this
one way or another.

For more about this class in $\pi_{45}S^0$, see the chart ``The $E_\infty$-page
of the classical Adams Spectral Sequence'' in \cite{Isaksen}. The fact
that some class detected by $h_3^2h_5$ can have a non-zero $\kappabar$
multiplication can be found in Lemma 4.114 of \cite{Isaksen2}. See also Table 33
of that paper. Note that Isaksen is
careful to label this lemma as tentative, as this is in the range where the homotopy
groups of spheres still need exhaustive study. In the table A.3.3
of \cite{RavVeryGreen}, a related class is posited to be detected in the Adams-Novikov Spectral
Sequence by the class $\gamma_4$, also of filtration $3$, but note the question mark there.
\end{rem}

\section{The mapping space argument}\label{ch:mapping}

We would like to extend the results of the Section \ref{ch:homotopy} in the following way. Let
$\iota:S^0 \to E^{hG_{48}}$ be the unit and $r$ the composition
$$
\pi_{192k+48} X \cong \pi_{192k+45}(\Sigma^{-3}X) \to \pi_{192k+45}E^{hG_{24}} \cong \FF_4
$$
defined in (\ref{diagramr}).  

\begin{thm}\label{mappingspace}The composite
$$
\xymatrix{
\pi_{192k+48}F(E^{hG_{48}},X) \rto^-{\iota_\ast} & \pi_{192k+48}X \rto^-r & \FF_4
}
$$
is surjective.
\end{thm}

We can use this result to build maps out of $E^{hG_{48}}$ as follows. Recall from 
(\ref{mappingf}) that if $F \subseteq \GG_2$ is a closed subgroup, then there is
an isomorphism $E^\ast[[\GG_2/F]]\cong E^\ast E^{hF}$ of $E^\ast[[\GG_2]]$-modules.
Also, Proposition \ref{dual-res} and the Universal Coefficient Theorem give an isomorphism
$E^\ast X \cong E^\ast[[\GG_2/G_{24}]]$. 

Consider the following
diagram. Note we are using that $E^{-48} = E_{48}$. 
\begin{equation}\label{diagr}
\xymatrix{
\pi_{48}(E^{hG_{48}},X) \rto^-{i_\ast}\dto_H & \pi_{48}X \dto^H\\
\Hom_{\GG_2}(E^{0}[[\GG_2/G_{24}]],E^{-48}[[\GG_2/G_{48}]]) \rto^-{\iota^\ast} &
\Hom_{\GG_2}(E^{0}[[\GG_2/G_{24}]],E^{-48}) \dto^{\cong}\\
&\WW[[j]]\Delta^2 \cong (E_{48})^{G_{24}} \dto^\epsilon\\
&\FF_4
}
\end{equation}
where the maps labelled $H$ are the Hurewicz maps for $E^\ast(-)$ and the map $\epsilon$ reduces 
mod $(2,j)$. By Corollary \ref{describer}, the  vertical composition on the right is $r$
up to some automorphism of $\FF_4$. Proposition \ref{thecrucialelement} and Corollary 
\ref{describer} then yield  the following corollary to Theorem \ref{mappingspace}, using the
case when $k=0$.

\begin{cor}\label{mapping2} Let $f(j)\Delta^2 \in H^0(G_{24},E_{48})$. Then there is a
map
$$
\phi:\Sigma^{48}E^{hG_{48}} \to X
$$
so that $\iota_\ast(\phi) \equiv f(0)$ modulo $2$.
\end{cor}

We will use this result to show that there is an equivalence $\Sigma^{48}E^{hG_{24}} \to X$. See
Theorem \ref{maintheorem} below.

We now begin the proof of Theorem \ref{mappingspace}. 
Let $p: E^{h\SS_2^1} \to E^{hG_{24}} \times E^{hG_{24}}$ be the projection
from the top to the bottom of the centralizer resolution tower.

\begin{lem}\label{thefirstcalculationredux} Let $k \in \ZZ$. The map
$$
p_\ast:\pi_{192k+45}F(E^{hG_{48}}, E^{h\SS_2^1}) \to \pi_{192k+45} (F(E^{hG_{48}},E^{hG_{24}}) 
\times F(E^{hG_{48}},E^{hG_{24}}))
$$
is surjective.
\end{lem}

\begin{proof} We apply $F(E^{hG_{48}},-)$ to the centralizer tower and examine the 
resulting spectral sequence in homotopy. See (\ref{cent-tower-ss}). The spectral
sequence reads
$$
\pi_tF(E^{hG_{48}},F_s) \Longrightarrow \pi_{t-s}F(E^{hG_{48}},E^{h\SS_2^1})
$$
and the fibers $F_s$ are described in (\ref{cent-fibers}). Thus we need to know 
\begin{align*}
0&=\pi_{192k+45}F(E^{hG_{48}},E^{hC_6}\vee E^{hC_4})\\
&= \pi_{192k+46}F(E^{hG_{48}},E^{hC_2})\\
&= \pi_{192k+47}F(E^{hG_{48}},E^{hC_6})\ .
\end{align*}
We can use (\ref{thegencase}). Note that $C_2$ is central, so all of the
subgroups $F_x$ contain $C_2$. Therefore, the crucial input is as before:
$$
\pi_kE^{hC_6} \subseteq \pi_kE^{hC_2} = 0
$$
for $k=45$, $46$, and $47$ and that $\pi_{45}E^{hC_4}=0$. See Propositions 
\ref{homotopyc2}, \ref{homotopyc6}, and \ref{homotopyc4}. See also Figure 2.
\end{proof}

Let $q: E^{h\SS_2^1} \to E^{hG_{24}}$ be the projection
from the top to the bottom of the duality resolution tower.
Using Remark \ref{compare-ss} we now can produce a commutative diagram, where we have abbreviated
$F(E^{hG_{48}},Y)$ as $F(Y)$ and we're writing $n=192k+45$.
\begin{equation}\label{bigd}
\xymatrix{
\pi_nF(\Sigma^{-3}X) \rto\dto_r &\pi_nF(E^{h\SS_2^1}) \rto^{=} \dto_{p_\ast}&
\pi_nF(E^{h\SS_2^1}) \dto^{q_\ast}\\
\pi_nF(E^{hG_{24}}) \rto \dto_{\iota_\ast} & \pi_n(F(E^{hG_{24}}) \times F(E^{hG_{24}}))
\rto^-{(g_0)_\ast} \dto_{\iota_\ast} &
\pi_nF(E^{hG_{24}}) \dto_{\iota_\ast}\\
\pi_n E^{hG_{24}} \rto & \pi_n(E^{hG_{24}}\times E^{hG_{24}}) \rto_-{(g_0)_\ast}  &\pi_n E^{hG_{24}}.
}
\end{equation}
The maps labelled $(g_0)_\ast$ are induced from the map $g_0:E^{hG_{24}} \times E^{hG_{24}}
\to E^{hG_{24}}$ discussed in Remark \ref{compare-ss}. It is projection onto the first factor. 
The lower two rows are split short exact, the maps labelled $\iota_\ast$ are all onto,
and the maps $p_\ast$ and $q_\ast$ are onto.  The map $r$ is then defined by the requirement that the upper right 
square commute. It is the analog of the map $r$ of (\ref{diagramr}), and in fact we have a commutative
diagram
$$
\xymatrix{\pi_{192k+45}F(E^{hG_{48}},\Sigma^{-3}X) \rto^-{\iota_\ast} \dto_r & \pi_{192k+45}\Sigma^{-3}X \dto^r\\
\pi_{192k+45}F(E^{hG_{48}},E^{hG_{24}}) \rto^-{\iota_\ast} & \pi_{192k+45}E^{hG_{24}} \cong \FF_4\ .
}
$$
Since $E^{hG_{24}}$ is an $E^{hG_{48}}$-module spectrum, the evaluation map
$$
\iota:F(E^{hG_{48}},E^{hG_{24}}) \to E^{hG_{24}}
$$
is split. Hence, 
Theorem \ref{mappingspace} follows from the next result. 

\begin{lem}\label{ragain} The map
$$
r:\pi_{192k+45}F(E^{hG_{48}},\Sigma^{-3}X) \to \pi_{192k+45}F(E^{hG_{48}},E^{hG_{24}})
$$
is onto.
\end{lem}

\begin{proof} The proof is an exact copy of the first part of the argument for Proposition 
\ref{thecrucialelement}, generalized to mappings spaces; that is, we examine the homotopy
spectral sequence built from the duality tower for $F(E^{hG_{48}},E^{h\SS_2^1})$.
In the diagram (\ref{bigd}), the map $q_\ast$ is onto, and any element
$x$ in the kernel of $(g_0)_\ast$ must have filtration at least $1$ in the homotopy spectral
sequence of the duality tower. Since $\pi_kF(E^{hG_{48}},E^{hC_6})=0$ for $k=46$ and $k=47$,
using (\ref{thegencase}) and Propositions 
\ref{homotopyc2} and \ref{homotopyc6},
any such element must be the image of class from $\pi_\ast F(E^{hG_{48}},\Sigma^{-3}X)$.
\end{proof}

\begin{rem}\label{instead} Note that all of these arguments would work with replacing
$E^{hG_{48}}$ with $E^{hG_{24}}$.
\end{rem}

The next result is our main theorem.

\begin{thm}\label{maintheorem} There is an equivalence 
$$
\Sigma^{48}E^{hG_{24}} \mathop{\longr}^{\simeq} X
$$
realizing the given isomorphism of Morava modules
$$
E_\ast E^{hG_{24}} \mathop{\longr}^{\cong} E_\ast X.
$$
\end{thm}

\begin{proof} We actually use the given isomorphism of Morava modules to produce
a (non-equivariant) equivalence
$$
C_2^+ \wedge \Sigma^{48}E^{hG_{48}}
 \simeq \Sigma^{48}(E^{hG_{48}} \vee E^{hG_{48}}) \to X.
$$
Here $C_2 = \Gal = \Gal(\FF_4/\FF_2)$. Then we will apply Lemma \ref{more-split}. 

We begin with some algebra. Recall from Remark \ref{all-the-splittings} that
the $2$-Sylow subgroup $S_2 \subseteq \SS_2$
can be decomposed as $K \rtimes Q_8$. Since $G_{24} = Q_8 \rtimes \FF_4^\times$ and
$G_{48} = G_{24} \rtimes \Gal$ we have that $E^\ast[[\GG_2/G_{24}]]$
and $E^\ast[[\GG_2/G_{48}]]$ are free $E^\ast[[K]]$ modules of rank $2$ and 
rank $1$ respectively. Since $K$ is a finitely generated pro-$2$-group, the
ring $E^0[[K]]$ is a complete local ring with maximal ideal
$\mm_K$ given by the kernel of the reduced augmentation
$$
E^0[[K]] \longr E^0 \longr \FF_4.
$$
We will produce a map
$$
f :   \Sigma^{48}(E^{hG_{48}} \vee E^{hG_{48}})  \to X
$$
so that the map of $E^0[[K]]$-modules
$$
E^{\ast}f: E^{48}X \longr E^{48}(\Sigma^{48}(E^{hG_{48}} \vee E^{hG_{48}}))
$$
is an isomorphism
modulo $\mm_K$. Then, by the appropriate variant of Nakayama's Lemma (see Lemma 4.3 of \cite{GHMR}),
it will be an isomorphism of $E^0[[K]]$-modules and, hence, of $E^0[[\GG_2]]$-modules, as required.

Since $G_{24}\cong Q_8\rtimes \mathbb{F}_4^{\times}$ and
$\mathbb{G}_2\cong((K\rtimes Q_8)\rtimes \mathbb{F}_4^{\times})\rtimes \Gal,$ we have an isomorphism
of  $K$-sets $\mathbb{G}_2/{G_{24}} \cong K\sqcup  K\phi$ where $\phi$ is the Frobenius in the Galois group. 
Then we have the following commutative diagram. It is an expansion of the diagram \eqref{diagr}.

$$
\xymatrix@C=15pt{
\Hom_{\GG_2}(E^0[[\GG_2/G_{24}]],E^{-48}[[\GG_2/G_{48}]]) \rto
 \dto_{\iota^\ast} & \Hom_K(E^0[[K]]\oplus E^0[[K \phi]],E^{-48}[[K]]) \dto^{\iota^\ast}\\
\Hom_{\GG_2}(E^0[[\GG_2/G_{24}]],E^{-48}) \rto \dto_\cong & \Hom_K(E^0[[K]]\oplus E^0[[K \phi]],E^{-48}) \dto^{\subseteq}\\
(E_{48})^{G_{24}} \cong \WW[[j]] \rto \dto & E^2_{48} \cong (\WW[[u_1]])^2 \dto\\
\FF_4 \rto &\FF_4^2\ .
}$$
The top two horizontal maps are forgetful maps, remembering only the $K$ action. The third horizontal map 
is the composition
$$
\xymatrix{
\WW[[j]] \rto^-\subseteq &\WW[[u_1]] \rto^-{1 \times \phi} & (\WW[[u_1]])^2.
}
$$
and the bottom map is $x \mapsto (x,\phi(x))$. 
The top vertical maps are induced by the
map $E^\ast(E^{hG_{48}}) \to E^\ast S^0$ given by the unit, the middle vertical map evaluates a
homomorphism at $1 \in E^0[[\GG_2/G_{24}]]$; note we are again
using that $E^{-48} \cong E_{48}$. The final map is reduction modulo the maximal
ideals in both cases. 

By Corollary \ref{mapping2} we can produce two maps
$$
f_i:\Sigma^{48}E^{hG_{48}} \to X,\quad i=1,2
$$
so that $(f_i)_\ast (\iota) \equiv \omega^i\Delta^2$ modulo $(2,j),$ where $\omega \in \mathbb{F}_4$ is the primitive cube root of unity. We now examine the fate
of
$$
E^0(f_i):E^0[[\GG_2/G_{24}]] \longr E^{-48}[[\GG_2/G_{48}]]
$$
as we work from the upper left to the bottom right of this diagram. Using the formulas of the previous paragraph we have
$$
E^0(f_1) \mapsto (\omega,\omega^2)\quad\mathrm{and}\quad E^0(f_2) \mapsto (\omega^2,\omega).
$$

Finally, let
$$
f = f_1 \vee f_2:  \Sigma^{48}(E^{hG_{48}} \vee E^{hG_{48}})  \to X.
$$
If we apply $E^\ast(-)$ to this map we get a map
$$
E^\ast f: E^0[[\GG_2/G_{24}]] \to E^{-48}[[\GG_2/G_{48}]] \times E^{-48}[[\GG_2/G_{48}]]
$$
which yields a map of $K$-modules
$$
E^\ast f: E^0[[K]]^2 \longr E^{-48}[[K]]^2
$$
which, modulo the maximal ideal $\mm_K$, gives the map
$$
\FF_4^2 \longr \FF_4^2
$$
given by the matrix
$$
\left( \begin{array}{ll}
\omega&\omega^2\\
\omega^2&\omega
\end{array}
\right)
$$
with determinant $\omega^2 + \omega = 1$. Thus $E^\ast f$ is an isomorphism, as needed. 
\end{proof}

We can now complete the proof of Theorem \ref{main} and construct the topological duality resolution.
Proposition  \ref{dual-res} and  Theorem  \ref{maintheorem} imply the following. Recall from
Theorem \ref{homotopyc6} that $E^{hC_6}$ is $48$-periodic. 

\begin{cor}
 There exists a resolution of $E^{h\SS_2^1}$ in the $K(2)$-local category at the prime 2
$$
E^{h\mathbb{S}_2^1} \to E^{hG_{24}} \to E^{hC_6} \to  \Sigma^{48}E^{hC_6} \to \Sigma^{48}E^{hG_{24}}\ .
$$
\end{cor}

\begin{rem}\label{notsphere} In \cite{GHMR}, working at the prime $3$, the authors were able to
produce a topological resolution for $L_{K(2)}S^0 = E^{h\GG_2}$ itself by the same methods
that produced the resolution for $E^{h\GG_2^1}$. The resolution for the sphere was essentially 
a double of that for $E^{h\GG_2^1}$. Not only do the methods of \cite{GHMR} not apply to the
case $p=2$, it's very unlikely that there is a topological resolution of the sphere, or 
for $E^{h\SS_2}$, which could be obtained by doubling the resolution of $E^{h\SS_2^1}$. There
are any number of difficulties, but the first obstacle is that we have only a semi-direct decomposition
$\SS_2^1 \rtimes \ZZ_2 \cong \SS_2$ at $p=2$, rather than the product decomposition $\GG_2^1 \times \ZZ_3 
\cong \GG_2$ at the prime $3$. This makes the algebra much harder, and it only gets worse from there.
Other short topological resolutions are possible, of course, and could be very instructive. This is
the subject of current research by Agn\`es Beaudry and Hans-Werner Henn.
\end{rem}

\bibliographystyle{alpha}

\bibliography{bib2}

\vspace{-0.2cm}
\end{document}